\documentclass[a4paper,12pt]{amsart}
\usepackage{amsmath,amssymb, amsthm}
\usepackage{appendix}
\usepackage{amsmath}
\usepackage{amssymb}
\usepackage{amscd}
\usepackage{latexsym}
\usepackage[dvipdfmx]{graphicx}
\usepackage{graphicx}
\usepackage{here}
\usepackage{xcolor}
\usepackage{palatino}
\newtheorem{theorem}{Theorem}[section]
\newtheorem{lemma}[theorem]{Lemma}
\newtheorem{proposition}[theorem]{Proposition}
\newtheorem{corollary}[theorem]{Corollary}
\newtheorem{assertion}[theorem]{Assertion}

\theoremstyle{definition}
\newtheorem{definition}[theorem]{Definition}
\newtheorem{example}[theorem]{Example}

\theoremstyle{remark}
\newtheorem{remark}[theorem]{Remark}

\numberwithin{equation}{section}

\newcommand{\m}{{\mathfrak m}}
\newcommand{\R}{{\mathbb R}}

\newcommand{\C}{{\mathbb C}}
\newcommand{\Z}{{\mathbb Z}}

\newcommand{\Hyp}{{\mathbb H}}
\newcommand{\Low}{{\mathbb L}}

\newcommand{\F}{{\mathcal F}}

\newcommand{\IE}{{\it i.e.}, }
\newcommand{\EG}{{\it e.g.}, }

\newcommand{\CP}{{\C}P}
\newcommand{\e}{\varepsilon}
\newcommand{\bfo}{{\bf{0}}}
\newcommand{\vphi}{\varphi }
\newcommand{\T}{^t\!}
\newcommand{\ds}{\displaystyle}
\title[Lefschetz fibrations]
{
{\large Lefschetz fibrations 
on the Milnor fibers of cusp and simple elliptic singularities}
}
\author[N Kasuya]{Naohiko Kasuya}
\address{Department of Mathematics, Faculty of Science, Hokkaido University, 
North 10, West 8, Kita-ku, Sapporo, Hokkaido 060-0810, Japan.}
\email{nkasuya@math.sci.hokudai.ac.jp}
\author[H Kodama]{Hiroki KODAMA}
\address{International Institute for Sustainability with Knotted
Chiral Meta Matter (SKCM2), Hiroshima University, 2-313 Kagamiyama,
Higashi-Hiroshima-shi, Hiroshima 739-0046, Japan.
\newline
\hspace*{\parindent}
Center for Interdisciplinary Theoretical and Mathematical Sciences
Program (iTHEMS), RIKEN, 2-1 Hirosawa, Wako-shi, Saitama 351-0198,
Japan}
\email{kodamahiroki@gmail.com}
\author[Y Mitsumatsu]{Yoshihiko Mitsumatsu}
\address{Department of Mathematics, Chuo University, 
1-13-27 Kasuga, Bunkyo-ku, 
Tokyo 112-8551, Japan}
\email{yoshi@math.chuo-u.ac.jp}
\author[A Mori]{Atsuhide MORI}
\address{Department of Mathematics, Osaka Dental University,
8-1 Kuzuha-Hanazono, Hirakata, Osaka 573-1121, Japan} 
\email{mori-a@cc.osaka-dent.ac.jp}

\begin{document}

\maketitle

\begin{abstract}
We show that the total space of the Milnor fibration 
associated with any cusp or simple elliptic singularity in complex three variables 
admits an $S^1$-parametric genus-one Lefschetz fibration structure over the $2$-disk. 
As a consequence, we demonstrate that 
the Lawson type foliations on $S^5$ associated with such singularities 
can be regarded as the pullback of the Reeb foliation on $S^3$. 
This enables us to provide an alternative proof 
of a previous result by the third author,  
which states that every Lawson type foliation admits a leafwise symplectic structure.  
Also we see that 
a pair of such Milnor fibers can be glued together 
along boundary into a closed oriented 4-manifold 
exactly when the pair corresponds to one of the 
ten extended strange duality pairs among the cusp singularities.  
This gluing is compatible with the Lefschetz fibrations 
and the resultant 4-manifold is diffeomrphic to a K3 surface.  
\end{abstract}
\setcounter{section}{-1}

\tableofcontents
\section{\large Introduction}

Our main result in the present article is the following.   
For the precise statement, see Theorem~\ref{precise statement}. 
\\

\noindent
{\bf Main Theorem}\quad 
{\em The Milnor fibration of a simple elliptic or cusp singularity 
in complex three variables admits an $S^1$-family of Lefschetz fibrations 
over the 2-disk whose regular fibers are diffeomorphic to the 2-torus. 
}
\\

In this article, we explicitly construct a map that realizes the Lefschetz fibration 
as a Lagrangian torus fibration. 
The map is, in fact, obtained as the restriction to each Milnor fiber of a map defined on $\C^3$. 
Consequently, all the Milnor fibers are simultaneously equipped with genus-one Lefschetz fibrations. 
This is what we mean by an ``$S^1$-family of Lefschetz fibrations."
On a single Milnor fiber of the cusp sigularities,  
the Lagrangian construction is 
also intrinsically given by Hacking and Keating \cite{HK} 
by arguments which are more of algebro-geometric nature than 
those in the present article.    
In order to construct such fibrations in an $S^1$-parametric way, 
we need a more topological method. 
\\

Now we consider a pair of cusp singularities whose monodromies, viewed as $T^2$-bundles monodromies of their links, are conjugate in $\mathit{SL}(2; \Z)$ to each other's inverses. 
In this case, the boundaries of the corresponding Milnor fibers are orientation-reversingly diffeomorphic to each other,  
so we obtain a closed orientable $4$-manifold by gluing them along their boundaries. 
\\

\noindent
{\bf  Theorem} (Smooth Decomposition of K3 Surface, Proposition~\ref{Prop:DualMonodromies} \& Theorem~\ref{Thm:MainCorollary})
\quad {\em 
For any two cusp singularities in complex three variables, 
the monodromies of the boundary $2$-torus bundles of the corresponding Milnor fibres 
are conjugate in $\mathit{SL}(2; \Z)$ to each other's inverses if and only if the singularities form an extended strange duality pair. 
Moreover, the $4$-manifold obtained by gluing the two Milnor fibers is diffeomorphic to a K3 surface for every such pair, regardless of the matching of their boundaries.} 
\\

The extended strange duality is the ten pairs of cusp singularities 
related to Arnold's strange duality among the exceptional unimodal singularities (see \S 4.2).  
It was noted in \cite{Mi2} that, 
in certain cases of the extended strange duality pairs, 
the two Milnor fibers can be glued together along their boundaries to become a closed symplectic 4-manifold 
by modifying their original exact symplectic structures to non-exact ones. 
Thus a natural problem was raised;   
\emph{What are these closed symplectic 4-manifolds?} 
Ue realized from the computation of 
the cohomology ring 
that the resultant manifold is homotopically equivalent to,  
and thus at least homeomorphic to a K3 surface. 
In this paper, we show that they are in fact diffeomorphic to 
a smooth K3 surface thanks to the Lefschetz fibrations.
\\

Like our construction of the $S^{1}$-parametric 
Lefschetz fibrations,  
our arguments on the decomposition of K3 surface 
are smooth topological. 
Each of so-called singular \emph{Hirzebruch-Inoue surfaces} 
has two cusp singularities.   
Nakamura (\cite{Nakamura1}, see also \cite{Nakamura2} and \cite{Lo}) 
found that each extended strange duality pair of cusp singularities 
appears as the two singularities on a single Hirzebruch-Inoue surface.  
He also showed that a singular Hirzebruch-Inoue surface 
admits a flat deformation to a K3 surface 
exactly when its two singularities form one of 
the extended strange duality pairs \cite{Nakamura1, Nakamura3}. 
The above smooth decomposition theorem captures  
the purely smooth topological aspect of this phenomenon.
In \S~\ref{subsec:Inose},  we give a possibly related 
example of a decomposition of a Kummer surface along a smoothly embedded Sol-manifold. 
\\

%
%

As another application, 
we give an alternative proof of the following theorem, which is 
due to the previous works  \cite{Mi1, Mi2} of the third author.  
This was the original motivation of the present work. 
\\

\noindent
{\bf Theorem}  \cite{Mi1, Mi2} \quad
{\em The Lawson type foliation of codimension one 
on the 5-sphere associated with 
a simple elliptic singularity or a cusp singularity 
admits a leafwise symplectic structure.  }
\\

As our construction in the Main Theorem 
is done on the total space of the Milnor fibrations, 
the existence of Lefschetz fibrations is obtained 
in an $S^{1}$-parametric way and 
implies the following theorem. 
Extending our construction to cusp singularities or 
simple elliptic singularities of complete intersections 
is left as an important problem in the future.  
\\

\noindent
{\bf Theorem} (Foliated Lefschetz Fibration, Theorem~\ref{Thm:FolLF})\quad 
{\em The Lawson type foliation on the 5-sphere 
which appears in the above theorem  
admits a foliated Lefschetz fibration structure 
over 
the standard Reeb foliation over the 3-sphere, 
with regular fibers diffeomorphic to the 2-torus.  }
\\

This result together with the foliated version of Gompf's theorem 
\cite{Gompf, GS} gives an alternative proof of the above theorem.  
For the precise definition of a foliated Lefschetz fibration,  
see Section \ref{Section;Foliated LF}. 
\\

In the proof of Main Theorem, 
our construction of Lefschetz fibrations  
is 
{explicit} to a certain degree.  
The basic model of the fibration map originates 
in the absolute value moment maps 
which describe certain contact structures, 
which is originally due to the fourth author A. Mori  
and developed by the first author N. Kasuya \cite{Kasuya} and R. Furukawa. 
What we need is a Lefschetz fibration with closed fibers on 
a compact Stein surface.  Therefore the fibration 
is not at all holomorphic 
with respect to the original complex structure of the Milnor fibers.  
Moreover, the total space is fixed, which is a Milnor fiber.  
Therefore we first look for a good candidate for the map 
and then rectify it to be a Lefschetz fibration.  
So we can not start from a holomorphic map as a model nor 
rely on some relations in the mapping class group of the fiber.  
One of the key steps is to confirm that the critical points 
of the constructed map are of genuine Lefschetz type.  
From this point of view, we obtained two different methods.  
In this article, we deform the model map to a Lagrangian torus 
fibration.  
Then Eliasson's work \cite{Eliasson} on the critical points 
in integrable Hamiltonian systems enables us to verify the critical points to be  
of Lefschetz type. 

The construction of the desired Lefschetz fibrations  
is also possible by analyzing the 2-jets of the model map at critical points 
and deforming it to a map 
so as to have the genuine Lefschetz type critical points. 
This method can be considered as a particular case  
of a study of the space of 2-jets 
of isolated critical points, which is explained together with the construction 
in a forthcoming paper \cite{K2M2}.  
\\

The article is organized as follows. 
In \S 1 
the main objects of the article, namely,  
the simple elliptic and cusp singularities are reviewed. 
Also the principal notion of the article, Lefschetz fibrations,  
is recalled. 
In \S 2 
we introduce the absolute value moment map 
and Lagrangian torus fibrations.  
Based on these preliminaries, the Main Theorem is precisely stated 
and is proved. 

Then in the subsequent sections, some applications are presented. 
The structure of the Milnor lattices 
and the monodromy of the Milnor fibrations of the singularities 
are described in \S 3 by looking at the Lefschetz fibration.  
This is a reformulation of the results by Gabrielov \cite{Gabrielov}.  

In \S 4 an application of the main results to K3 surfaces is presented. 
First, the strange duality and Hirzebruch-Inoue surfaces are reviewed,  
then 
Theorem (Smooth Decomposition of K3 Surface)  
is presented in more detail. 
Also, a decomposition of 
the Inose fibrationan, an elliptic fibration of a K3 surface, is constructed. 
In \S 5 the application of the main result to the Lawson type foliations 
on the 5-sphere is presented. 
\\

\noindent
Acknowledgements:

The authors are grateful to Francisco Presas for suggesting 
to prove the existence of leafwise symplectic structure 
by that of foliated Lefschetz fibration.  
Also, they are grateful to Masaaki Ue for 
lots of important information and suggestions 
on the topology of elliptic surfaces.

\section{\large Singularities and Lefschetz fibrations} 
Throughout this paper, $(x, y, z)$ denote the coordinates on $\C^3$. 
\subsection{Simple elliptic and cusp singularities}~\label{sing}

First we recall some special types of isolated hypersurface singularities in $\C^3$. 
Each of the following polynomials have the only singularity at the origin $\bfo$. 
\begin{eqnarray*}
\tilde{E_6} &\colon & x^3+y^3+z^3+axyz, \; a^3+27\ne 0,\\
\tilde{E_7} &\colon & x^2+y^4+z^4+axyz, \; a^4-64\ne 0,\\
\tilde{E_8} &\colon & x^2+y^3+z^6+axyz, \; a^6-432\ne 0. 
\end{eqnarray*}
These singularities are called simple elliptic singularities. 
The following polynomial also defines an isolated singularity at the origin, which is called a cusp singularity:   
\begin{eqnarray*}
T_{pqr} \colon  x^p+y^q+z^r+axyz, \; a \ne 0,\; \dfrac{1}{p}+\dfrac{1}{q}+\dfrac{1}{r}<1. 
\end{eqnarray*} 
Simple elliptic singularities and cusp singularities are of different types  
in the sense of Arnold's classification of hypersurface singularities. 
However, when the parameter $a$ is a sufficiently large positive number, 
they can be summarized into the single form 
\begin{eqnarray*}
x^p+y^q+z^r+axyz, \;\; \dfrac{1}{p}+\dfrac{1}{q}+\dfrac{1}{r}\leq 1. 
\end{eqnarray*}

Next we review the precise definitions. 
Generally, simple elliptic and cusp singularities need not to be hypersurface singularities. 
They are formally defined by using the minimal resolution and its exceptional set. 

\begin{definition}[simple elliptic singularity]
Let $(S, \bfo)$ be a normal surface singularity and $\pi \colon \tilde{S}\to S$ its minimal resolution. 
$(S, \bfo)$ is called a simple elliptic singularity if the exceptional set $E=\pi^{-1}(\bfo)$ is an elliptic curve. 
\end{definition}

We note that the link of the singularity is diffeomorphic to 
the boundary of a tubular neighborhood of the exceptional set. 
Hence, the link of an simple elliptic singularity is diffeomorphic to 
the circle bundle over the $2$-torus with the Euler class $-k$, 
where $-k$ is the self-intersection number of the elliptic curve $E$. 
In other words, it is diffeomorphic to the $T^2$-bundle over the circle 
with the monodromy $\begin{pmatrix} 1&0 \\ k&1\end{pmatrix}$. 
The following theorem shows when a simple elliptic singularity becomes a hypersurface singularity. 

\begin{theorem}[Saito \cite{S}]
A simple elliptic singularity can be embedded in $\C^3$   
if and only if it is analytically equivalent to $\tilde{E_6}, \tilde{E_7}$ or $\tilde{E_8}$. 
\end{theorem}

Since they satisfy $E^2=-3$, $-2$ and $-1$, 
their links are $T^2$-bundles over the circle with the monodromy 
$\begin{pmatrix} 1&0 \\ 3&1\end{pmatrix}$, 
$\begin{pmatrix} 1&0 \\ 2&1\end{pmatrix}$ and 
$\begin{pmatrix} 1&0 \\ 1&1\end{pmatrix}$, respectively. 

\begin{definition}[cusp singularity]
Let $(S, \bfo)$ be a normal surface singularity and $\pi \colon \tilde{S}\to S$ its minimal resolution. 
$(S, \bfo)$ is called a cusp singularity if the exceptional set $E=\pi ^{-1}(\bfo)$ is 
a cycle $C=C_1+\cdots +C_n$ of non-singular rational curves $C_i$ ($1\leq i \leq n$, $n\geq 2$) 
or a single rational curve $C_1$ with a node.   
\end{definition}

Here a cycle means that if $n\geq 3$, for any $i$ with $1\leq i\leq n$,  
$C_i$ intersects with $C_{i+1}$ at only one point transversely and $C$ has no other crossing, 
where $C_{n+1}$ denotes $C_1$, and if $n=2$, $C_1$ and $C_2$ intersect transversely at distinct two points. 
Now we set $b_i=-C_i^2$ if $n\geq 2$, and $b_1=2-C_1^2$ if $n=1$. 
Then it follows that $b_i\geq 2$ for all $i$ and $b_i\geq 3$ for some $i$, 
since the intersection matrix $(C_iC_j)_{1\leq i, j\leq n}$ must be negative definite by Grauert's criterion.   
Then the link of a cusp singularity is diffeomorphic to 
the $T^2$-bundle over the circle with the hyperbolic monodromy 
$$A=\begin{pmatrix}0&1\\-1&b_1\end{pmatrix}\cdots\begin{pmatrix}0&1\\-1&b_n\end{pmatrix}.$$ 
It is also known which cusp singularities are realized as hypersurface singularities in 3 variables. 

\begin{theorem}[Karras \cite{Kar}]
A cusp singularity can be embedded in $\C^3$ 
if and only if it is analytically equivalent to one of $T_{pqr}$. 
\end{theorem}

By calculating the minimal resolution and plumbing along the cyclic graph, 
we can see that the link of $T_{pqr}$ singularity is diffeomorphic to 
the $T^2$ bundle over the circle with the hyperbolic monodromy
$$A_{p,q,r}=\begin{pmatrix}r-1&-1\\1&0\end{pmatrix}
\begin{pmatrix}q-1&-1\\1&0\end{pmatrix}
\begin{pmatrix}p-1&-1\\1&0\end{pmatrix},$$ 
(see \cite{Lau}, \cite{Neumann}, \cite{EW}, \cite{Kasuya}). 
Moreover, Neumann showed the following characterization of simple elliptic and cusp singularities. 
\\

\begin{theorem}[\cite{Neumann}]~\label{chara}
A singularity link fibers over $S^1$ if and only if it is either the link of a simple elliptic or cusp singularity. 
\end{theorem}

\subsection{Milnor's fibration theorem}\label{MilnorFibration}

Let $(z_1, \ldots, z_n)$ be the coordinates on $\C^n$ and $f(z_1, \ldots, z_n)$ 
be a polynomial of complex $n$-variables with isolated critical point 
at the origin $\bfo$. Then, the zero level set $V(0):=f^{-1}(0)$ is an algebraic variety with isolated singularity at $\bfo$.  
For a sufficiently small positive number $\e $, 
the sphere $S^{2n-1}_{\e}$ of radius $\e$ centered 
at the origin transversely intersects with $V(0)$. 

\begin{definition}[Milnor radius, singularity link]
The Milnor radius $\e _f$ of $f$ is the supremum of such $\e$'s, namely, 
$$\e_f=\sup \left\{\e \; \middle| \;  S^{2n-1}_{r} \text{ is transverse to } 
V(0) \text{ for any $0<r \leq \e$}\right\}. $$
The intersection $L:=V(0)\cap S^{2n-1}_{\e}$ ($\e <\e _f$) 
is called the link of the singularity. 
\end{definition}

\begin{theorem}[Milnor \cite{M}]~\label{Milnor}
For any positive number $\e$ with $\e<\e_f$, the map 
$$\frac{f}{|f|}\colon S^{2n-1}_{\e}\setminus L\to S^1$$ 
is a fiber bundle over the circle. 
Moreover, there exists 
a positive number $\delta $ such that if  
$0\leq |t| \leq \delta $, then $V(t):=f^{-1}(t)$ 
transversely intersects with $S^{2n-1}_{\e}$. 
For such $\e$ and $\delta $, the map 
$$f|_{f^{-1}(S^1_{\delta })\cap D^{2n}_{\e }}\colon 
f^{-1}(S^1_{\delta })\cap D^{2n}_{\e }\to S^1_{\delta }$$
is also a fiber bundle over the circle, which is isomorphic to $\frac{f}{|f|}$ as a fiber bundle. 
\end{theorem}

\begin{definition}[Milnor fibration, Milnor fiber, Milnor tube]
The fiber bundle $f|_{f^{-1}(S^1_{\delta })\cap D^{2n}_{\e }}$ 
is called the Milnor fibration of $f$, 
and each fiber 
$$F_{\theta }:=f^{-1}(\delta e^{i\theta })\cap D^{2n}_{\e}$$ 
is called its Milnor fiber. 
Moreover, $f^{-1}(D^2_{\delta })\cap D^{2n}_{\e}$ 
is called the Milnor tube of $f$. 
\end{definition}

\subsection{The Milnor fibers of simple elliptic and cusp singularities}~\label{Milnor fiber}
In the following, we put 
$$f(x,y,z)=x^p+y^q+z^r+axyz \; 
(a\ne 0, \; \frac{1}{p}+\frac{1}{q}+\frac{1}{r}\leq 1), $$
$M=\max \left\{p, q, r \right\}$ and $V_a(\e, w)=f^{-1}(w)\cap D^6_{\e}$. 
By \cite{Kasuya-Mori}, if $a$ is a positive real number greater than $M$, 
then the Milnor radius is greater than $1$. 
Thus we may assume that $\e=1$. 
Now we prove the following lemma to estimate the size of the Milnor tube. 
\vspace{4pt}
\begin{lemma}~\label{thm: tube}
Let the positive real number $a$ and the complex number $t$ satisfy the conditions $a>12M$ and $0<|w|<1$. 
Then, $V_a(1,w)$ is a Milnor fiber. 
\end{lemma}
\begin{proof}
We assume that $f(x,y,z)=t$ and $|x|^2+|y|^2+|z|^2=1$. 
Then the gradient vector 
$$\nabla f=(px^{p-1}+ayz, \; qy^{q-1}+azx, \; rz^{r-1}+axy)$$
and the vector $(\bar{x}, \bar{y}, \bar{z})$ are linearly independent over $\C$. 
We prove this claim by reduction to absurd, namely, 
we assume that the two vectors $\nabla f$ and 
$(\bar{x}, \bar{y}, \bar{z})$ are linearly dependent, and lead a contradiction. 
Considering symmetry, we may assume that 
$|x|\geq |y| \geq |z|$, in particular, $|x|\geq \sqrt{\frac{1}{3}}$.

First, by triangle inequality, we obtain 
$$|ayz|=\Big|x^{p-1}+\frac{y}{x}y^{q-1}+\frac{z}{x}z^{r-1}-\frac{t}{x} 
\Big|\leq 1+1+1+\sqrt{3}<5, $$
in particular, $|z|<\sqrt{\frac{5}{a}}$. 
By the assumption on the linear dependence of 
$\nabla f$ and $(\bar{x}, \bar{y}, \bar{z})$, 
we have 
$$|x||axy+rz^{r-1}|=|z||ayz+px^{p-1}|.$$ 
Then it follows from triangle inequality that 
$$a|x^2y|\leq a|yz^2|+p|x^{p-1}z|+r|xz^{r-1}|,$$ and hence, 
\begin{eqnarray}
a(|x|^2-|z|^2)|y|\leq (p|x|^{p-2}+r|z|^{r-2})|xz|. 
\end{eqnarray}
On the other hand, by 
$\sqrt{\frac{1}{3}}\leq |x|\leq 1$, $|z|^2<\frac{5}{a}$ and $a>12M>30$, 
the inequality 
\begin{eqnarray*}
a(|x|^2-|z|^2)> 12M(\frac{1}{3}-\frac{1}{6})
=2M \geq p+r \geq (p|x|^{p-2}+r|z|^{r-2})|x|
\end{eqnarray*}
holds. 
Hence, together with the inequality (1.1), $y$ must be zero. 
Since $|y|\geq |z|$ and $|x|^2+|y|^2+|z|^2=1$, we have $|x|=1$ and $y=z=0$.
However, for such a point $(x,0,0)$, 
$$|f(x,0,0)|=|x^p|=1>|w|$$holds, hence $f(x,0,0)=w$ cannot be satisfied. 
This is a contradiction. 
Therefore, 
$\nabla f$ and $(\bar{x}, \bar{y}, \bar{z})$ 
are linearly independent over $\C$. 
\end{proof}

Now we take a positive number $m$, and retake $a$ such that 
$$a>\max \left\{ 12M, m^2(m+3)\right\}. $$
Then the following lemma holds. 
\vspace{4pt}
\begin{lemma}~\label{thm: ball}
For any real number $\theta $, $V_a(1,\frac{1}{a}e^{i\theta})$ 
is a Milnor fiber, and any point $(x,y,z)$ on $V_a(1,\frac{1}{a}e^{i\theta})$ 
satisfies $\max \left\{ |x|, |y|, |z|\right\}>\frac{m}{a}$. 
\end{lemma}
\begin{proof}
The former claim is obvious from Lemma~\ref{thm: tube}. 
Now we put $\rho =\frac{m}{a}$. 
If $|x|, \; |y|, \; |z|\leq \rho $, then it follows that  
\begin{eqnarray*}
|x^p+y^q+z^r+axyz|\leq \rho^p+\rho^q+\rho^r+a\rho^3\leq 3\rho^2+a\rho^3
\\
=\frac{m^2}{a^2}\Big(3+\frac{am}{a}\Big)<\frac{m^2(m+3)}{a^2}<\frac{1}{a}, 
\end{eqnarray*}
and hence, $(x,y,z)$ is not on $V_a(1,\frac{1}{a}e^{i\theta})$. 
\end{proof}

\subsection{Lefschetz fibration} 
From the Morse lemma, given a holomorphic function $\widetilde{G}\colon U\subset \C^2\to \C$ with a non-degenerate critical point $0$, we can take a holomorphic coordinate system $(u,v)$ on a small neighborhood of $0$ such that the restriction of $\widetilde{G}$ is expressed as $\widetilde{G}(0)+u^2+v^2$. A Lefschetz singularity is in short a critical point of a real smooth map modelled on the complex Morse singularity. 
Let $M^4$ be an oriented $4$-manifold, and $\Sigma$ an oriented surface. 

\begin{definition}[Lefschetz singularity, Lefschetz fibration] 
\quad 
\begin{enumerate}
\item
A critical point $P$ of a smooth map $G\colon M^4 \to \Sigma$ is called a Lefschetz singularity 
if there exist positively oriented coordinate systems $(u_1,u_2,v_1,v_2)$ and $(w_1,w_2)$ respectively defined near $P\in M^4$ and $G(P)\in \Sigma$ such that the map $G$ is locally expressed as 
$\widetilde{G}\colon (u_1,u_2,v_1,v_2)\mapsto (w_1,w_2)$ with $w_1+i w_2=(u_1+i u_2)^2+(v_1+iv_2)^2$. 
\item
If $G\colon M^4\to \Sigma$ is a fibration with isolated critical points 
of a compact manifold $M^4$ and all its critical points are Lefschetz singularities, 
we call $G$ a Lefschetz fibration.  
\item
Two Lefschetz fibrations $G\colon M \to \Sigma$ and $G'\colon M' \to {\Sigma}'$ 
are isomorphic if there exist orientation preserving diffeomorphisms 
$\Phi \colon M \to M'$ and $\varphi \colon \Sigma \to \Sigma'$ satisfying 
$\varphi\circ G=G'\circ\Phi$.
\end{enumerate}
\end{definition}

A Lefschetz singularity is locally presented by a non-degenerate homogeneous quadratic map. 
We notice that the converse is not true even when a given 
quadratic map is homotopic to one presenting a Lefschetz singularity 
through non-degenerate homogeneous ones. 
We would like to inform that the space of such quadratic maps are investigated in a forthcoming paper \cite{K2M2}. 

For a Lefschetz fibration $G\colon M^4\to \Sigma$, 
the preimage $G^{-1}(c)$ of a critical value $c$ is called a singular fiber at $c$. 
We usually assume that no singular fiber contains a $2$-sphere that is embedded in $M^4$ with self-intersection $-1$. This condition is called the relative minimality.  
Removing all the singular fibers, we obtain a fiber bundle over a punctured surface 
$\Sigma'=\Sigma\setminus\{c_1,\dots,c_n\}$. 
A fiber of this bundle is called a regular fiber, and its genus is called the genus of the Lefschetz fibration. 
In a standard way to understand the topology of a Stein manifold, 
Ailsa Keating \cite{Ke} used a Lefschetz fibration over $D^2$ whose regular fiber 
has non-empty boundary in her recent work on the topology of a Milnor fiber. 
In this paper we construct a Lefschetz fibration over $D^2$ whose regular fiber 
is the torus $T^2$ instead, and consider it as a ``half'' of the elliptic fibration of a K3 surface. 
To this aim we restrict ourselves to the case where the regular fiber is $T^2$.  

Take a small disk $D_j$ on $\Sigma$ centered at $c_j$ with radius $\e>0$ 
with respect to the local coordinates ($j=1,\dots n$). 
Then, on the local model, the intersection of $\widetilde{G}^{-1}(\e,0)$ and 
$\R^2=\langle(1,0,0,0),(0,0,1,0)\rangle$ is a loop which bounds a disk on $\R^2$. 
This defines a loop on a regular fiber near 
$c_j$ which is called the vanishing cycle of the singular fiber at $c_j$. 
In the case where a singular fiber contains $p$ critical points, 
the vanishing cycles of $c_j$ are $p$ parallel loops. 
The monodromy of the $T^2$-bundle along $\partial D_j$ is then (isotopic to)
the composition of the right-handed Dehn-twists along the loops. 
We call it the monodromy of the singular fiber at $c_j$. 
If the fibration is generic, that is, if each singular fiber has a single critical point, 
the monodromy is a single Dehn-twist. 
The total monodromy of a Lefschetz fibration over $D^2$ is the monodromy along $\partial D^2$.
If it is trivial, we can obtain a Lefschetz fibration over $S^2$ 
by attaching $T^2\times D^2$ along the boundary. 
If we attach $T^2\times \Sigma _{g, 1}$ instead, 
we can also obtain a Lefschetz fibration over $\Sigma _g$, where $\Sigma _g$ and $\Sigma _{g, 1}$ are 
compact orientable genus $g$ surfaces with no boundary and with one boundary component, respectively. 

\begin{example}
The right-handed Dehn-twists $\tau_\alpha$, $\tau_\beta$ along the standard generator $\alpha$, $\beta$ of $\pi_1(T^2)$ satisfy the relation $(\tau_\beta\circ\tau_\alpha)^{6n}=1$. 
Taking the left-hand side as the total monodromy, we obtain a genus-one Lefschetz fibration $f_n$ over $S^2$ with $12n$ critical points, whose total space is usually denoted by $E(n)$. 
Similarly, we can obtain a genus-one Lefschetz fibration $h_{g, n}$ over $\Sigma_g$ with $12n$ critical points.  
\end{example}

In fact, genus-one Lefschetz fibrations over $S^2$ have been classified by Kas and Moishezon as follows.  

\vspace{4pt}

\begin{theorem}[Kas \cite{Kas}, Moishezon \cite{Moishezon}]~\label{Kas-Moishezon}
Let $f\colon M^4\to S^2$ be a relatively minimal genus-one Lefschetz fibration with at least one singular fiber. Then it is isomorphic to the Lefschetz fibration $f_n\colon E(n)\to S^2$ for some positive integer $n$. \\
\end{theorem}

\begin{remark}~\label{E2}
It is well known that all the K3 surfaces are diffeomorphic to each other. 
Moreover, a generic elliptic fibration of an elliptic K3 surface over $\C P^1$ has exactly $24$ Lefschetz critical points, 
since the Euler characteristic of a K3 surface is equal to $24$. 
Hence, it follows from Theorem~\ref{Kas-Moishezon} that any K3 surface is diffeomorphic to $E(2)$. 
\end{remark}

Theorem~\ref{Kas-Moishezon} was generalized by Matsumoto to the following result, 
which completely classifies genus-one Lefschetz fibrations over any closed orientable surface. 
The argument by Matsumoto is similar to that by Moishezon, 
but is more-arranged and clarifies the topological meaning of the proof 
even in the restricted case where the base is $S^2$. 

\vspace{4pt}

\begin{theorem}[Matsumoto \cite{Ma}]~\label{Matsumoto}
Let $f\colon M^4\to \Sigma _g$ be a relatively minimal genus-one Lefschetz fibration with at least one singular fiber. Then it is isomorphic to the Lefschetz fibration $h_{g, n}$ for some positive integer $n$. 
\end{theorem}

\section{\large Construction of Lefschetz fibration }~\label{CLF}

Let $g\colon \C^3\to \C$ be the complex-valued function defined by 
$$g(x,y,z)=|x|^2+e^{\frac{2\pi i}{3}}|y|^2+e^{\frac{4\pi i}{3}}|z|^2. $$ 
This function arises from the moment map of $\C^3$ (see Example~\ref{moment map}). 
In this section, we prove the following theorem, which provides the precise formulatoin of our main result. 

\vspace{5pt}

\begin{theorem}~\label{precise statement}
Let $X$ be the Milnor fiber of a simple elliptic or cusp singularity. 
Then there exists a smooth homotopy $X_t$ $(0\leq t\leq 1)$ of convex symplectic submanifolds of $\C^3$ 
such that $X_0=X$ and $g|_{X_1}$ is a Lagrangian torus fibration 
that fibers the convex boundary $\partial X_1$ by regular tori 
and has exactly $(p+q+r)$ singular points, all of which are of Lefschetz type.
\end{theorem}

\begin{remark} 
Recall that the Milnor fiber $X$ in Theorem~\ref{precise statement} can be written as $$X=V_a(1, w)=f^{-1}(w)\cap D^6,$$ 
where $a>12M$ and $w$ is any complex number with $0<|w|<1$ (see Lemma~\ref{thm: tube}). 
In the proof of Theorem~\ref{precise statement}, we will construct the homotopy $X_t$ 
as the level sets of some homotopy $\{f_t\}_{0\leq t \leq 1}$ of functions with $f_0=f$, 
which does not depend on the value of $w$. 
Hence, if we set $w=\dfrac{1}{a} e^{i\theta }$, then we obtain the smooth family 
$\{ g|_{f_1^{-1}(\frac{1}{a} e^{i\theta })\cap D^6} \}$ 
of Lagrangain torus fibrations parametrized by $\theta \in S^1$. 
This implies that the function $g$ yields the structure of $S^1$-parametric genus-one Lefschetz fibrations 
on the total space of the Milnor fibration $\frac{f}{|f|}\colon S^5\setminus L\to S^1$. 
\end{remark}
 
Now, as a preliminary to the proof, we review the properties of the function $g$ for a while. 
The singular set $\Sigma (g)$ is easily determined to be the union of $x$-axis, $y$-axis and $z$-axis. 
In $\C^3\setminus \Sigma (g)$, the level set of $g$ is a real $4$-dimensional submanifold. 
Now we want to describe the tangent space at a regular point $(x,y,z)$. 
First we define real vector fields $e_x, e_y, e_z, E_x, E_y, E_z$ by 
\begin{align*}
&e_x=i(x\frac{\partial}{\partial x}-\bar{x}\frac{\partial }{\partial \bar{x}}), \; 
e_y=i(y\frac{\partial}{\partial y}-\bar{y}\frac{\partial }{\partial \bar{y}}), \; 
e_z=i(z\frac{\partial}{\partial z}-\bar{z}\frac{\partial }{\partial \bar{z}}),  \\
&E_x=x\frac{\partial}{\partial x}+\bar{x}\frac{\partial }{\partial \bar{x}}, \; 
E_y=y\frac{\partial}{\partial y}+\bar{y}\frac{\partial }{\partial \bar{y}}, \; 
E_z=z\frac{\partial}{\partial z}+\bar{z}\frac{\partial }{\partial \bar{z}}. 
\end{align*}
Moreover, we define a real vector field $e_0$ by 
$$e_0=\frac{1}{|x|^2}E_x+\frac{1}{|y|^2}E_y+\frac{1}{|z|^2}E_z.$$
Notice that $e_x, e_y, e_z, E_x, E_y, E_z$ are defined on $\C^3$ while $e_0$ is defined only on $(\C^{\ast})^3$. 
By using these vector fields, a basis of the tangent space is described as follows: 
\begin{enumerate}
\item
$\{ e_x, \; e_y, \; e_z, \; e_0 \}$ if $xyz\ne 0$, 
\item
$\left\{ e_y, \; e_z, \; 
\dfrac{\partial }{\partial x}+\dfrac{\partial }{\partial \bar{x}}, \; 
i(\dfrac{\partial}{\partial x}-\dfrac{\partial }{\partial \bar{x}})\right\}$ if $x=0$ and $yz\ne 0$, 
\item
$\left\{ e_z, \; e_x, \; 
\dfrac{\partial }{\partial y}+\dfrac{\partial }{\partial \bar{y}}, \; 
i(\dfrac{\partial}{\partial y}-\dfrac{\partial }{\partial \bar{y}})\right\}$ if $y=0$ and $zx\ne 0$, 
\item
$\left\{ e_x, \; e_y, \; 
\dfrac{\partial }{\partial z}+\dfrac{\partial }{\partial \bar{z}}, \; 
i(\dfrac{\partial}{\partial z}-\dfrac{\partial }{\partial \bar{z}})\right\}$ if $z=0$ and $xy\ne 0$.
\end{enumerate}

In the following two subsections, we make a brief review of integrable Hamiltonian systems, 
and then, give the proof of Theorem~\ref{precise statement} in the last subsection. 

\subsection{The moment maps and Lagrangian torus fibrations}

Let $(M^{2n}, \omega )$ be a symplectic $2n$-manifold and $h, h' \colon M^{2n}\to \R$ smooth functions. 
The Hamiltonian vector field $V_h$ of $h$ is defined by the equation $\iota _{V_h}\omega =-dh$, 
and the Poisson bracket of two functions $h$ and $h'$ is defined by $$\{ h, h' \}=\omega (V_h, V_{h'}). $$ 
\begin{definition}[integrable Hamiltonian system]
A pair of a symplectic $2n$-manifold $(M^{2n}, \omega )$ and 
a smooth map $$H=(h_1, \ldots , h_n)\colon M^{2n}\to \R^n$$ is called an integrable Hamiltonian system 
if $\{ h_i, h_j \}=0$ for all pairs $i, j$ with $1\leq i, j \leq n$ and $dH$ is surjective in some open dense set of $M^{2n}$. 
\end{definition}

Now we assume that an integrable Hamiltonian system 
$(M^{2n}, \omega , H)$ has only compact and connected fibers.
Then, according to Arnold-Liouville theorem, a regular fiber of $H$ is a Lagrangian $n$-torus. 
Moreover, its neighborhood is symplectomorphic to $(D^n\times T^n, \sum_{i=1}^n dp_i\wedge dq_i)$, 
where $p_i$'s and $q_i$'s are the coordinates on $D^n$ and $T^n$, respectively, 
and each fiber of $H$ is written as the Lagrangian torus $\{ \ast \}\times T^n$. 
Hence, we have the Hamiltonian $T^n$-action, 
and after composing with an appropriate coordinate transformation of $D^n$, 
we can call $H\colon M^{2n}\to \R^n$ the moment map. 

\begin{example}[the moment map for $\C^n$]~\label{moment map}
We define a $T^n$-action on $\C^n$ by 
$$(t_1, \ldots, t_n)\cdot (z_1, \ldots , z_n)=(e^{it_1}z_1, \ldots , e^{it_n}z_n).$$
Then it is a Hamiltonian $T^n$-action and its moment map $\mu \colon \C^n \to \R^n$ is given by 
$$\mu (z_1, \ldots , z_n)=\frac{1}{2}(|z_1|^2, \ldots , |z_n|^2).$$
In particular, when $n=3$, the moment map $\mu \colon \C ^3\to \R^3$ is written as 
$$\mu (x, y, z)=\frac{1}{2}(|x|^2, |y|^2, |z|^2).$$
\end{example}

\begin{remark}
With a slight abuse of terminology, we also call the restriction of $\mu $ to $S^{2n-1}$ the moment map. 
When $n=3$, the moment map $\mu |_{S^5}$ is very useful 
for analyzing the links of simple elliptic and cusp singularities (see \cite{Kasuya}). 
As we will see in \S~\ref{construction}, 
a similar strategy works for studying the Milnor fibers of simple elliptic and cusp singularities. 
Namely, the function $g$, which is a variant of the moment map $\mu |_{S^5}$, 
plays an important role in the construction of a Lagrangian torus fibration with Lefschetz singularities. 
\end{remark}

\subsection{Non-degenerate singularities in integrable Hamiltonian systems}

\begin{definition}[non-degenerate singularity of maximal corank]
A singular point $x_0$ of an integrable Hamiltonian system 
$(M^{2n}, \omega , H)$ is called non-degenerate of maximal corank if 
$$dh_1(x_0)=\cdots =dh_n(x_0)=0$$ and the quadratic parts of $h_1, \ldots , h_n$ at $x_0$ generate 
a Cartan subalgebra $\mathcal{C}$ of the algebra of quadratic forms $\mathrm{Q}(2n)$ under the Poisson bracket. 
\end{definition}

\vspace{4pt}

\begin{theorem}[Eliasson \cite{Eliasson}]~\label{Eliasson}
Let $(M^{2n}, \omega , H)$ be an integrable Hamiltonian system, 
and $x_0$ its non-degenerate singular point of maximal corank. 
Then there exist a local symplectomorphism  $\Phi \colon (T_{x_0}M^{2n}, 0)\to (M^{2n}, x_0)$ 
and smooth functions $\psi _1, \ldots , \psi _n$ such that 
$$h_i\circ \Phi =\psi _i(q_1, \ldots , q_n) \; \; (1\leq i\leq n), $$ 
where $q_1, \ldots , q_n$ is the basis of $\mathcal{C}$. 
\end{theorem}

Namely, in integrable Hamiltonian systems, 
non-degenerate singularities of maximal corank can be determined only by their $2$-jets. 
In particular, when $n=2$, they are classified into the following four types: 
\begin{align*}
&h_1=x_1^2+y_1^2, \; h_2=x_2^2+y_2^2 \;\;\; \text{(elliptic case)}, \\
&h_1=x_1^2+y_1^2, \; h_2=x_2y_2 \;\;\; \text{(elliptic-hyperbolic case)}, \\
&h_1=x_1y_1, \; h_2=x_2y_2 \;\;\;  \text{(hyperbolic case)}, \\
&h_1=x_1y_1+x_2y_2, \; h_2=x_1y_2-x_2y_1 \;\;\;  \text{(focus-focus case)},
\end{align*}
where $(x_1, y_1, x_2, y_2)$ is the coordinates on $\R^4$ 
equipped with the standard symplectic structure $dx_1\wedge dy_1+dx_2\wedge dy_2$.
Here we said ``focus-focus'' following the terminology in the area of Hamiltonian systems. 
However, the singularity is nothing but a Lefschetz singularity, so hereafter, 
we use the word ``Lefschetz'' instead of ``focus-focus''. \\

\subsection{Construction of a Lagrangian torus fibration of the Milnor fiber}~\label{construction}
\\

We suppose that $f$, $M$, $a$ satisfy 
the same conditions as in \S~\ref{Milnor fiber}, and $g$ as defined at the beginning of \S~\ref{CLF}. 
Moreover, we fix the number $m$ by $m=30M$. 
In the following, we give a deformation $\left\{ X_t \right\}_{t\in [0, 1]}$ 
of  the Milnor fiber $X_0=V_a(1,\frac{1}{a}e^{i\theta })$ as a convex symplectic submanifold of $\C^3$ 
such that $g|_{X_1}$ is a Lagrangian torus fibration with only Lefschetz singularities. 
In order to do so, we construct a deformation $\left\{ f_t \right\}_{t\in [0, 1]}$ of the function $f_0=f$. 

First, we take a bump function $\vphi \colon \R_{\geq 0}\cup \left\{ \infty \right\}\to [0, 1]$ that satisfies the conditions 
\begin{eqnarray*}
\vphi (s)\equiv 1 \; \big(0\leq s\leq \frac{1}{6}\big), \;  
\vphi (s)\equiv 0 \; \big(\frac{1}{2}\leq s\big), \;  
-4\leq \vphi '(s) \leq 0.
\end{eqnarray*}
Then we define the functions $\vphi _1, \vphi _2, \vphi _3$ by  
\begin{eqnarray*}
\vphi_1(|x|, |y|, |z|)=\vphi (\frac{\sqrt{|y|^2+|z|^2}}{|x|}), \\ 
\vphi_2(|x|, |y|, |z|)=\vphi (\frac{\sqrt{|z|^2+|x|^2}}{|y|}), \\
\vphi_3(|x|, |y|, |z|)=\vphi (\frac{\sqrt{|x|^2+|y|^2}}{|z|}). 
\end{eqnarray*}
Notice that $\vphi _1, \vphi _2, \vphi _3$ cannot be defined at the origin $\bfo$, 
so these are the functions defined on $(\C^3)^{\ast}$. 
The supports of these functions have no intersection each other. 
Hence, for any $(x,y,z)\in (\C^3)^{\ast}$, at least two of the three vanish at the point. 

Using these bump functions, 
we define the function $h\colon (\C^3)^{\ast} \to \C$ by 
$$h(x,y,z)=\vphi _1x^p+\vphi _2y^q+\vphi _3z^r+axyz$$ 
and give the deformation $f_t \; (0\leq t \leq 1)$ by $f_t=(1-t)f+th$. 
Finally, we set $$X_t=f_t^{-1}(\frac{1}{a}e^{i\theta})\cap D^6_{1}.$$ 
In order to prove Theorem~\ref{precise statement}, 
we first show the following properties of the functions $\vphi _1$, $\vphi _2$ and $\vphi_3$. 
\vspace{4pt}
\begin{lemma}~\label{phi }
The equalities and inequalities
\begin{eqnarray*}
e_x(\vphi _j)=0, \; e_y(\vphi _j)=0, \; e_z(\vphi _j)=0 \; (j=1,2,3), 
\end{eqnarray*}
\begin{eqnarray*}
e_0(\vphi _1)=\dfrac{2|x|^2-|y|^2-|z|^2}{|x|^3\sqrt{|y|^2+|z|^2}}\vphi '\big(\frac{\sqrt{|y|^2+|z|^2}}{|x|}\big), \\
e_0(\vphi _2)=\dfrac{2|y|^2-|z|^2-|x|^2}{|y|^3\sqrt{|z|^2+|x|^2}}\vphi '\big(\frac{\sqrt{|z|^2+|x|^2}}{|y|}\big), \\
e_0(\vphi _3)=\dfrac{2|z|^2-|x|^2-|y|^2}{|z|^3\sqrt{|x|^2+|y|^2}}\vphi '\big(\frac{\sqrt{|x|^2+|y|^2}}{|z|}\big),  
\end{eqnarray*}
\begin{eqnarray*}
\| \nabla \vphi _1 \|=\| \overline {\nabla } \vphi _1 \|<\frac{3}{|x|}, \; 
\| \nabla \vphi _2 \|=\| \overline {\nabla } \vphi _2 \|<\frac{3}{|y|}, \; 
\| \nabla \vphi _3 \|=\| \overline {\nabla } \vphi _3 \|<\frac{3}{|z|}
\end{eqnarray*}
hold if the both sides are defined. 
\end{lemma}
\begin{proof}
Since $\vphi _j$ is a function of $|x|$, $|y|$, $|z|$, 
it is preserved by the rotation vector fields $e_x$, $e_y$, $e_z$. 
Hence, we have $e_x(\vphi _j)=e_y(\vphi _j)=e_z(\vphi _j)=0$. 

Since $\vphi _j$ is a real-valued function, $\overline {\nabla } \vphi _j$ is the complex conjugation of $\nabla \vphi _j$. 
Hence, $\| \nabla \vphi _j \|=\| \overline {\nabla } \vphi _j \|$.
By explicit computations of derivatives, we have 
\begin{eqnarray*}
\frac{\partial \vphi_1}{\partial x}&=&-\frac{\sqrt{|y|^2+|z|^2}}{|x|^2}
\vphi'\Big(\frac{\sqrt{|y|^2+|z|^2}}{|x|}\Big)\frac{\partial |x|}{\partial x}
=-\frac{\sqrt{|y|^2+|z|^2}}{2|x|x}
\vphi'\Big(\frac{\sqrt{|y|^2+|z|^2}}{|x|}\Big), \\
\frac{\partial \vphi_1}{\partial y}
&=&\frac{|y|}{|x|\sqrt{|y|^2+|z|^2}}\vphi'\Big(\frac{\sqrt{|y|^2+|z|^2}}{|x|}\Big)
\frac{\partial |y|}{\partial y}
=\frac{\bar{y}}{2|x|\sqrt{|y|^2+|z|^2}}
\vphi'\Big(\frac{\sqrt{|y|^2+|z|^2}}{|x|}\Big),\\
\frac{\partial \vphi_1}{\partial z}
&=&\frac{|z|}{|x|\sqrt{|y|^2+|z|^2}}\vphi'\Big(\frac{\sqrt{|y|^2+|z|^2}}{|x|}\Big)
\frac{\partial |z|}{\partial z}
=\frac{\bar{z}}{2|x|\sqrt{|y|^2+|z|^2}}
\vphi'\Big(\frac{\sqrt{|y|^2+|z|^2}}{|x|}\Big).
\end{eqnarray*}
Hence, 
\begin{eqnarray*}
\| \nabla \vphi _1 \|^2
&=&
\Bigl|\frac{\partial \vphi_1}{\partial x}\Big|^2
+\Big|\frac{\partial \vphi_1}{\partial y}\Big|^2
+\Big|\frac{\partial \vphi_1}{\partial z}\Big|^2\\
&=&\frac{1}{4|x|^2}\Big(\frac{|y|^2+|z|^2}{|x|^2}+1\Big) \Big(\vphi'\Big(\frac{\sqrt{|y|^2+|z|^2}}{|x|}\Big)\Big)^2\\
&<&\frac{1}{4|x|^2}\cdot \frac{5}{4}\cdot 4^2=\frac{5}{|x|^2}<\frac{9}{|x|^2}. 
\end{eqnarray*}
Then it follows that 
$$\| \nabla \vphi _1 \|=\| \overline {\nabla } \vphi _1 \|<\frac{3}{|x|}$$ if $x\ne 0$.  
Similarly, 
$$\| \nabla \vphi _2 \|=\| \overline {\nabla } \vphi _2 \|<\frac{3}{|y|} \; \text{ if $y\ne 0$}, \;\;
\| \nabla \vphi _3 \|=\| \overline {\nabla } \vphi _3 \|<\frac{3}{|z|} \; \text{ if $z\ne 0$}. $$

When $xyz\ne 0$, $e_0$ is defined and $e_0(\vphi _j)$ can be computed as follows: 
\begin{eqnarray*}
e_0(\vphi _1)&=&
\frac{1}{|x|^2}(x\frac{\partial \vphi _1}{\partial x}+\bar{x}\frac{\partial \vphi _1}{\partial \bar{x}})
+\frac{1}{|y|^2}(y\frac{\partial \vphi _1}{\partial y}+\bar{y}\frac{\partial \vphi _1}{\partial \bar{y}})
+\frac{1}{|z|^2}(z\frac{\partial \vphi _1}{\partial z}+\bar{z}\frac{\partial \vphi _1}{\partial \bar{z}})\\
&=&\Big(-\frac{\sqrt{|y|^2+|z|^2}}{|x|^3}+\frac{1}{|x|\sqrt{|y|^2+|z|^2}}+\frac{1}{|x|\sqrt{|y|^2+|z|^2}}\Big)
\vphi'\Big(\frac{\sqrt{|y|^2+|z|^2}}{|x|}\Big) \\
&=&\frac{2|x|^2-|y|^2-|z|^2}{|x|^3\sqrt{|y|^2+|z|^2}}\vphi'\Big(\frac{\sqrt{|y|^2+|z|^2}}{|x|}\Big). 
\end{eqnarray*}
Similarly, 
\begin{eqnarray*}
e_0(\vphi _2)=\dfrac{2|y|^2-|z|^2-|x|^2}{|y|^3\sqrt{|z|^2+|x|^2}}\vphi '\big(\frac{\sqrt{|z|^2+|x|^2}}{|y|}\big), \\
e_0(\vphi _3)=\dfrac{2|z|^2-|x|^2-|y|^2}{|z|^3\sqrt{|x|^2+|y|^2}}\vphi '\big(\frac{\sqrt{|x|^2+|y|^2}}{|z|}\big). 
\end{eqnarray*}
\end{proof}

Now we are ready to prove the following theorem. 
\vspace{4pt}

\begin{theorem}~\label{thm: main}
For each $t\in [0,1]$, $X_t$ is a convex symplectic submanifold of $\C^3$. 
In particular, $\left\{ X_t \right\}_{t\in [0, 1]}$ is a homotopy of Liouville manifolds. 
\end{theorem}
\begin{proof}
Where $\vphi _1, \vphi _2, \vphi _3$ all vanish, 
the defining function $$f_t=(1-t)(x^p+y^q+z^r)+axyz$$ is holomorphic. 
Then, in such a region, $X_t$ is obviously a symplectic submanifold of $\C^3$. 
Since at least two of $\vphi _1, \vphi _2, \vphi _3$ vanish,  
it is enough to prove that $\mathrm{supp}(\vphi _1)\cap X_t$ is a symplectic submanifold of $\C^3$. 
Moreover, if the condition $$\sqrt{|y|^2+|z|^2}\leq \frac{|x|}{6}$$ is satisfied, 
there the defining function $f_t$ is again holomorphic. 
Hence, in the following, we argue under the assumptions 
$$\frac{|x|}{6}<\sqrt{|y|^2+|z|^2}<\frac{|x|}{2}, \;\; h(x,y,z)=\vphi _1x^p+axyz. $$

Now we prove the inequality $\| \nabla f_t \|>\|\overline {\nabla }f_t \|$ on $X_t$. 
The holomorphic and anti-holomorphic gradient vectors of $f_t$ are described as follows;
\begin{eqnarray*}
\nabla f_t&=&(1-t)\nabla f+t\nabla h\\
&=&(1-t)\begin{pmatrix} px^{p-1}+ayz \\ qy^{q-1}+azx \\ rz^{r-1}+axy \end{pmatrix}
+t\begin{pmatrix} p\vphi _1x^{p-1}+ayz \\ azx \\ axy\end{pmatrix}
+tx^p\nabla \vphi _1\\
&=&a\begin{pmatrix} yz \\ zx \\ xy \end{pmatrix}
+\begin{pmatrix} (1-t+t\vphi _1)px^{p-1} \\ (1-t)qy^{q-1} \\ (1-t)rz^{r-1} \end{pmatrix}
+tx^p\nabla \vphi _1, \\
& & \\
\overline {\nabla }f_t &=&(1-t)\overline {\nabla } f+t\overline {\nabla } h
=t\overline {\nabla } h=tx^p\overline {\nabla } \vphi _1. 
\end{eqnarray*}
Since the same argument as the proof of Lemma~\ref{thm: ball} works on $X_t$, 
for any $(x,y,z)\in \mathrm{supp}(\vphi _1)\cap X_t$, we have the inequality 
$$|x|=\max \left\{ |x|, |y|, |z|\right\}>\frac{m}{a}=\frac{30M}{a}.$$
Then it follows that 
\begin{eqnarray*}
\| \nabla f_t \| -\| \overline {\nabla }f_t  \|
&>&a|x|\sqrt{|y|^2+|z|^2}-\sqrt{3}M|x|-2t|x|^p\| \nabla \vphi _1 \| \\
&>&\frac{a}{6}|x|^2-(\sqrt{3}M+6)|x|=\Big(\frac{a|x|}{6}-(\sqrt{3}M+6)\Big)|x|\\
&>&(5M-\sqrt{3}M-6)|x|>3(M-2)|x|>0. 
\end{eqnarray*}
Hence, each $X_t$ is a symplectic submanifold of $\C^3$. 
Moreover, the gradient vector field of the squared distance function 
$(|x|^2+|y|^2+|z|^2) |_{X_t}$ restricted to $X_t$ is Liouville and outward transverse to the boundary $\partial X_t$. 
Therefore, $X_t$ is a convex symplectic submanifold of $\C^3$, and in particular, a Liouville manifold. 
\end{proof}

\vspace{4pt}

\begin{theorem}~\label{thm: main 2}
For each $t\in [0,1]$, the map $g|_{X_t}\colon X_t \to \C$ has exactly $(p+q+r)$ critical points 
$$\big(a^{-\frac{1}{p}}e^{\frac{i\theta}{p}}{(u_p)}^j, 0, 0\big), \; \big(0, a^{-\frac{1}{q}}e^{\frac{i\theta}{q}}{(u_q)}^k, 0\big), \; \big(0,0, a^{-\frac{1}{r}}e^{\frac{i\theta}{r}}{(u_r)}^l\big),$$
where $u_n=\exp {(\frac{2\pi i}{n})}$, 
$0\leq j \leq p-1$, $0\leq k \leq q-1$, and $0\leq l \leq r-1$. 
Moreover, $g|_{X_1}$ is a Lagrangian torus fibration whose singularities are of Lefschetz type. 
\end{theorem}
\begin{proof}
First, we determine all the critical points of the function $g|_{X_t}$. 
Suppose that the point $(x,y,z)$ in $X_t$ satisfies $xyz\ne 0$. 
Then it is a regular point of $g|_{X_t}$. This is proved as follows. 

Recall that when $xyz\ne 0$, 
the tangent space of the level set of $g$ is spanned over $\R $ by the four vectors $e_x, e_y, e_z$ and $e_0$. 
Then it is enough to show that the level sets of $f_t$ and $g$ are transversal at $(x,y,z)$, 
and so let us prove $$\langle e_x(f_t), e_y(f_t), e_z(f_t), e_0(f_t) \rangle_{\R}=\C.$$
By symmetry, we may assume that $|x|\geq |y|\geq |z|>0$. 
Then $f_t$ can be described as $$f_t=(1-t+t\vphi _1)x^p+(1-t)(y^q+z^r)+axyz. $$
By Lemma~\ref{phi } and explicit computations, we obtain the following: 
\begin{eqnarray*}
e_x(f_t)&=&i\Big((1-t+t\vphi _1)px^p+axyz\Big), \\
e_y(f_t)&=&i\Big((1-t)qy^q+axyz\Big), \\
e_z(f_t)&=&i\Big((1-t)rz^r+axyz\Big), \\
e_0(f_t)&=&
\Big(\dfrac{(1-t+t\vphi _1)p}{|x|^2}+t\dfrac{2|x|^2-|y|^2-|z|^2}{|x|^3\sqrt{|y|^2+|z|^2}}
\vphi '\big(\frac{\sqrt{|y|^2+|z|^2}}{|x|}\big) \Big)x^p \\
& &+(1-t)(\frac{qy^q}{|y|^2}+\frac{rz^r}{|z|^2})+(\dfrac{1}{|x|^2}+\dfrac{1}{|y|^2}+\dfrac{1}{|z|^2})axyz.
\end{eqnarray*}
Then $e_z(f_t)$ is nonzero, since $|axyz|$ is greater than $|(1-t)rz^r|$. 
Indeed, we have 
$$\Big|\frac{(1-t)rz^r}{axyz}\Big|=(1-t)|z|^{r-2}\frac{|z|}{|y|}\frac{r}{a|x|}<\frac{M}{30M}=\frac{1}{30}.$$
Hence, the argument of the nonzero complex number $e_z(f_t)$ can be estimated as follows: 
\begin{eqnarray}
\Big| \arg {\Big( \frac{e_z(f_t)}{ixyz} \Big)} \Big| <\arcsin {\big(\frac{1}{30}\big)}<\frac{\pi }{6}.
\end{eqnarray}
Similarly, $e_0(f_t)$ is nonzero because of the following inequalities. 
\begin{eqnarray*}
\Bigg|\Big(\dfrac{(1-t+t\vphi _1)p}{|x|^2}+t\dfrac{2|x|^2-|y|^2-|z|^2}{|x|^3\sqrt{|y|^2+|z|^2}}
\vphi '\big(\frac{\sqrt{|y|^2+|z|^2}}{|x|}\big) \Big)x^p+(1-t)(\frac{qy^q}{|y|^2}+\frac{rz^r}{|z|^2})\Bigg|\\ 
<(p+48)|x|^{p-2}+q|y|^{q-2}+r|z|^{r-2}<3M+48, 
\end{eqnarray*}
\begin{eqnarray*}
\Big| (\dfrac{1}{|x|^2}+\dfrac{1}{|y|^2}+\dfrac{1}{|z|^2})axyz \Big|
>(\frac{1}{|y|^2}+\frac{1}{|z|^2})|axyz|=\Big(\frac{|z|}{|y|}+\frac{|y|}{|z|}\Big)a|x| \geq 2a|x|>60M. 
\end{eqnarray*}
Here we used the estimates $a|x|>30M$ and 
\begin{eqnarray*}
\Bigg| \dfrac{2|x|^2-|y|^2-|z|^2}{|x|\sqrt{|y|^2+|z|^2}}
\vphi '\big(\frac{\sqrt{|y|^2+|z|^2}}{|x|}\big) \Bigg|
&<& \frac{4\cdot 2|x|}{\sqrt{|y|^2+|z|^2}} <48 \;\;\;\;  \text{if $\frac{|x|}{6}<\sqrt{|y|^2+|z|^2}$}, \\
&=&0 \hspace{100pt}  \text{otherwise}. 
\end{eqnarray*}
Moreover, the argument of $e_0(f_t)$ is estimated as 
\begin{eqnarray}
\Big| \arg {\Big( \frac{e_0(f_t)}{xyz} \Big)} \Big| <\arcsin {\Big(\frac{3M+48}{60M}\Big)}<
\arcsin {\big(\frac{1}{3}\big)}<\frac{\pi }{6}.
\end{eqnarray}
Then, by (2.1) and (2.2), the two complex numbers $e_z(f_t)$ and $e_0(f_t)$ are linearly independent over $\R$, 
which implies $$\langle e_x(f_t), e_y(f_t), e_z(f_t), e_0(f_t)\rangle _{\R}=\C.$$ Hence, $(x,y,z)$ is a regular point of $g|_{X_t}$.

Next we suppose $(x,y,z)\in X_t$ satisfies $xyz=0$. 
We want to show that if only one of $x$, $y$ and $z$ is zero, then $(x,y,z)$ is a regular point of $g|_{X_t}$. 
By symmetry we may assume that $z=0$ and $(x,y,0)$ is outside $\mathrm{supp}(\vphi _2)$. 
Now recall that the tangent space of $g^{-1}(\ast )$ at such a point is spanned over $\R$ 
by the four vectors $$e_x, \; e_y, \; 
\dfrac{\partial }{\partial z}+\dfrac{\partial }{\partial \bar{z}}, \;
i(\dfrac{\partial }{\partial z}-\dfrac{\partial }{\partial \bar{z}}).$$
Since $\frac{\partial \vphi _1}{\partial z}|_{z=0}=\frac{\partial \vphi _1}{\partial \bar{z}}|_{z=0}=0$, 
it follows that 
\begin{eqnarray*}
e_x(f_t)|_{z=0}&=&i(1-t+t\vphi _1)px^p, \\
e_y(f_t)|_{z=0}&=&i(1-t)qy^q, \\
(\frac{\partial }{\partial z}+\frac{\partial }{\partial \bar{z}})(f_t)|_{z=0}&=&axy, \\
i(\dfrac{\partial }{\partial z}-\dfrac{\partial }{\partial \bar{z}})(f_t)|_{z=0}&=&iaxy. 
\end{eqnarray*}
When $xy$ is nonzero, $axy$ and $iaxy$ are linearly independent over $\R$. 
Thus, these four complex numbers span $\C$ over $\R$, which means that the point $(x,y,0)$ is a regular point of $g|_{X_t}$. 
Therefore, all the critical points of $g|_{X_t}$ lie on the union of the $x$-axis, the $y$-axis and the $z$-axis. 
Then, considering the defining equation of $X_t$, the possible critical points are 
$$
\big(a^{-\frac{1}{p}}e^{\frac{i\theta}{p}}{(u_p)}^j, 0, 0\big), \; \big(0, a^{-\frac{1}{q}}e^{\frac{i\theta}{q}}{(u_q)}^k, 0\big), \; \big(0,0, a^{-\frac{1}{r}}e^{\frac{i\theta}{r}}{(u_r)}^l\big),
$$
all of which are indeed critical points of $g|_{X_t}$. 

Finally, we show that $g|_{X_1}$ is a Lagrangian torus fibration whose singularities are of Lefschetz type. 
First, $g|_{X_1}$ is a torus fibration, since all the fibers are non-singular tori over the region where $h(x,y,z)=axyz$ holds. 
Next we show that the fibers of $g|_{X_1}$ are Lagrangian submanifolds. 
For this purpose, it is enough to argue on the open dense subset $X_1\cap \left\{xyz\ne 0 \right\}$ of $X_1$. 
Again by symmetry, we may assume that $|x|\geq |y|\geq |z|$. 
In this case, the defining function of $X_1$ is described as $$f_1(x,y,z)=h(x,y,z)=\vphi _1x^p+axyz.$$
Then we obtain 
\begin{eqnarray*}
e_x(h)&=&i\big(\vphi _1px^p+axyz\big), \\
e_y(h)&=&iaxyz, \\
e_z(h)&=&iaxyz, \\
e_0(h)&=&
\Big(\frac{\vphi _1p}{|x|^2}+\dfrac{2|x|^2-|y|^2-|z|^2}{|x|^3\sqrt{|y|^2+|z|^2}}\vphi '\big(\frac{\sqrt{|y|^2+|z|^2}}{|x|}\big) \Big)x^p \\
& &+(\dfrac{1}{|x|^2}+\dfrac{1}{|y|^2}+\dfrac{1}{|z|^2})axyz. 
\end{eqnarray*}
Recall that $X_1=h^{-1}(\frac{1}{a}e^{i\theta})$ and the tangent space of $g^{-1}(\ast )$ is spanned by $e_x$, $e_y$, $e_z$ and $e_0$. 
Since we have $(e_y-e_z)(h)=0$, we can take a basis of the tangent space of a fiber of $g|_{X_1}$ of the form 
$$e_y-e_z, \; b_0e_0+b_1e_x+b_2e_y+b_3e_z, $$
where $b_0, b_1, b_2, b_3$ are some real constants that depends on the point $(x,y,z)$. 
Substituting them into the standard symplectic form $\omega _0$ on $\C^3$, we obtain 
\begin{eqnarray*}
\omega _0 \big(b_0e_0+b_1e_x+b_2e_y+b_3e_z, e_y-e_z\big)=b_0\omega _0(e_0, e_y)-b_0\omega _0(e_0, e_z) \\
=\frac{b_0}{|y|^2}\omega _0\big(E_y, e_y\big)-\frac{b_0}{|z|^2}\omega _0\big(E_z, e_z\big)
=b_0-b_0=0, 
\end{eqnarray*}
which implies that the fiber of $g|_{X_1}$is a Lagrangian submanifold of $X_1$. 
Moreover, as is proved in Lemma~\ref{Hessian} below,  
the Hessian at each critical point coincides with that of Lefschetz singularity. 
Then, by Theorem~\ref{Eliasson}, it is indeed a Lefschetz singularity. 
Therefore, $g|_{X_1}$ is indeed a Lagrangian torus fibration whose 
$(p+q+r)$ critical points are all Lefschetz singularities. 
\end{proof}

Now we prove the following lemma to complete the proof of Theorem~\ref{thm: main 2}. 
\vspace{4pt}
\begin{lemma}~\label{Hessian}
The $2$-jet of each critical point of the map $g|_{X_1}\colon X_1\to \C$ is $\mathcal{A}$-equivalent to that of the Lefschetz singularity. 
\end{lemma}

\begin{proof}
We put $x_0=(\frac{1}{a})^{\frac{1}{p}}e^{\frac{i\theta}{p}}$ and $s=x-x_0$. 
It is enough to show that the critical point $(x_0, 0, 0)$ has the same $2$-jet as that of the Lefschetz singularity.  
The Hessian at the critical point $(x_0, 0, 0)$ can be computed as follows. 
\begin{eqnarray*}
x^p+axyz=x_0^p &\Longleftrightarrow& a(s +x_0)yz=-s(px_0^{p-1}+\cdots +px_0s^{p-2}+s^{p-1})\\
&\Longrightarrow & yz=-\frac{px_0^{p-2}}{a}s-\frac{p(p-3)x_0^{p-3}}{2a}s^2+O(|s|^3) \\
&\Longrightarrow & s=-\frac{a}{px_0^{p-2}}yz
-\frac{(p-3)a^2}{2p^2x_0^{2p-3}}(yz)^2+O\big(|yz|^3\big), 
\end{eqnarray*}
and hence, 
\begin{eqnarray*}
|x|^2&=&(s+x_0) \overline{(s+x_0)} \\
&= & 
\big(x_0-\frac{a}{px_0^{p-2}}yz-\cdots \big)
\big(\overline{x}_0-\frac{a}{p\overline{x}_0^{p-2}}\overline{yz}-\cdots \big)\\
&= & |x_0|^2-2\mathrm{Re}\big(\frac{a\overline{x}_0}{px_0^{p-2}}yz\big)+\frac{a^2}{p^2|x_0|^{2(p-2)}}|yz|^2+\cdots .
\end{eqnarray*}
Taking local coordinates $v$ and $w$ on $X_1$ as the restrictions of $y$ and 
$\displaystyle\frac{\overline{x_0}|x_0|^{p-3}}{x_0^{p-2}}z$, 
respectively, 
we can express the map $g|_{X_1}$ as 
$$(v,w) \mapsto 
a^{-\frac{2}{p}}-\frac{1}{2}(|v|^2+|w|^2)-\lambda\mathrm{Re}(vw)
+\frac{\sqrt{3}i}{2}(|v|^2-|w|^2)+\cdots ,$$
where $\displaystyle \lambda =\frac{2}{p}a^{\frac{2p-3}{p}}>1$. 
Hence, in the real coordinates $(v_1, v_2, w_1, w_2)$, where $v=v_1+iv_2$ and $w=w_1+iw_2$,   
the Hessians $A$ and $B$ of the real and imaginary parts of $g|_{X_1}$ are    
$$A
=\begin{pmatrix} 
-1&0&-\lambda &0 \\ 
0&-1&0&\lambda \\
-\lambda &0&-1&0\\
0&\lambda &0&-1
\end{pmatrix}, \; 
B
=\sqrt{3}\begin{pmatrix} 
1&0&0&0\\
0&1&0&0\\
0&0&-1&0\\
0&0&0&-1
\end{pmatrix},$$
respectively. Then, taking the conjugates of them by the orthogonal matrix 
$\displaystyle P=\frac{1}{\sqrt{2}}\begin{pmatrix} 1&1&0&0 \\ 0&0&1&1 \\ -1&1&0&0 \\ 0&0&1&-1 \end{pmatrix}$, we have 
$$\T{P}AP
=\begin{pmatrix} \lambda -1 &0&0&0 \\ 0&-\lambda-1&0&0 \\ 0&0&\lambda -1&0 \\ 0&0&0&-\lambda-1 \end{pmatrix}, 
\; \; 
\T{P}BP
=\sqrt{3}\begin{pmatrix} 0&1&0&0 \\ 1&0&0&0 \\ 0&0&0&1 \\ 0&0&1&0 \end{pmatrix}.$$
Therefore, the $2$-jet of the singularity $(x_0, 0, 0)$ of $g|_{X_1}$ is $\mathcal{A}$-equivalent to 
the Lefschetz singularity, which is given by 
$$v^2+w^2=(v_1+iv_2)^2+(w_1+iw_2)^2=(v_1^2-v_2^2+w_1^2-w_2^2)+2i(v_1v_2+w_1w_2).$$
\end{proof}

Thus, we have obtained $X_1$ Liouville homotopic to 
the original Milnor fiber, together with a Lagrangian torus fibration $g|_{X_1}$ on it. 
However, there remains one issue to address. 
The boundary $\partial X_1$ is not fibered by regular fibers of $g|_{X_1}$, 
so we need to take a Liouville domain $Y$ in $X_1$ 
such that $Y$ is Liouville homotopic to $X_1$ 
and the boundary $\partial Y$ is fibered by Lagrangian tori. 
This is accomplished by the following lemma, 
which provides a smooth homotopy of Liouville domains $\{X_t\}_{0\leq t\leq 2}$ such that $X_0=X$ and $X_2=Y$. 
By reparametrizing the homotopy so that $X_0=X$ and $X_1=Y$, we complete the proof of Theorem~\ref{precise statement}. 

\begin{lemma}
Suppose that $a>\max \left\{3^M, m^2(m+3) \right\}$, where $m=30M$. 
Then $$Y=X_1\cap g^{-1}\Big(D^2_{\frac{1}{3}}\Big)$$ is a Liouville domain that is Liouville homotopic to $X_1$. 
Moreover, the boundary $\partial Y$ is fibered by regular fibers of $g|_{X_1}$. 
\end{lemma}

\begin{proof}
By the condition $a>3^M$, all the critical values of $g|_{X_1}$ 
are inside the disk $\displaystyle D^2_{\frac{1}{9}}=\left\{w\in \C \; \middle| \; |w|\leq \frac{1}{9} \right\}$. 
Indeed, $a^{-\frac{2}{p}}$, $a^{-\frac{2}{q}}$, $a^{-\frac{2}{r}}$ 
are all equal or less than $a^{-\frac{2}{M}}$ which is less than $\frac{1}{9}$. 
Then the last claim is obvious by construction. \\
Now we put $\displaystyle Z=X_1\cap D^6_{\frac{1}{2}}$. 
Then $Z$ contains all the singular fibers of $g|_{X_1}$ by the following argument. 
Let $\Delta (r)$ be the triangle in $\C$ whose vertices are $r$, 
$re^{\frac{2\pi i}{3}}$ and $re^{\frac{4\pi i}{3}}$. 
If $(x,y,z)\in \partial Z$, 
then $g(x,y,z)$ is contained in the neighborhood of width $\frac{1}{4050}$ 
of the boundary edges of the triangle $\Delta(\frac{1}{4})$. 
Indeed, we have 
\begin{eqnarray*}
\min \left\{ |x|^3, |y|^3, |z|^3 \right\} \, \leq \, |xyz|
&=&\frac{1}{a}\Big|-x^p-y^q-z^r+\frac{1}{a} e^{i\theta} \Big|\\
&<&\frac{1}{a}  <  \frac{1}{m^2(m+3)}  <  (\frac{1}{90})^3,
\end{eqnarray*}
and $g(x,y,z)\in \partial \Delta (\frac{1}{4})$ if and only if $xyz=0$. 
Similarly, if $(x,y,z)\in \partial X_1$, 
then $g(x,y,z)$ is in a small neighborhood of the boundary edges of the triangle $\Delta (1)$. 
Moreover, the image of $X_1-\mathrm{Int} Z$ by $g$ is disjoint from the disk $\displaystyle D^2_{\frac{1}{9}}$. 
Hence, $Z$ contains all the singular fibers. \\
Then $X_1- \mathrm{Int}Z$ is the symplectization of the convex hypersurface $\partial Z$ 
with the Liouville vector field $$v=\nabla (\rho^2|_{X_1-\mathrm{Int} Z}),$$
where $\rho=\sqrt{|x|^2+|y|^2+|z|^2}$. By a similar argument above, 
we can easily see that $g(Z)$ is contained in the disk $\displaystyle D^2_{\frac{1}{3}}$, and $g(\partial X_1)$ is outside the disk. 
This implies that $\partial Y\subset \mathrm{Int} X_1-Z$. 
Moreover, the Liouville vector field $v$ is transversal to $\partial Y$. This is proved as follows.  \\
Since $X_1=h^{-1}(\frac{1}{a}e^{i\theta})$ is a sympletic submanifold of  $\C^3$, 
the gradient vector field $\nabla \rho^2$ splits into the sum $v+v'$, 
where $v$ is tangent to $X_1$ and $v'$ is symplectically orthonormal to $X_1$. 
We note that $\nabla \rho^2=E_x+E_y+E_z$, and so, 
$$dg(\nabla \rho^2)=dg(E_x+E_y+E_z)=2(|x|^2+e^{\frac{2\pi i}{3}} |y|^2+e^{\frac{4\pi i}{3}}|z|^2)=2g. $$
Hence, it is enough to show that the inequality $|dg(v')|<\frac{2}{3}$ holds along $\partial Y$. 
In order to do so, we first describe a basis of the symplectic orthonormal vector bundle $(TX_1)^{\perp}$. 
Since 
$$dh=(\frac{\partial h}{\partial x}dx+\frac{\partial h}{\partial y}dy+\frac{\partial h}{\partial z}dz)+(\frac{\partial h}{\partial \bar{x}}d\bar{x}+\frac{\partial h}{\partial \bar{y}}d\bar{y}+\frac{\partial h}{\partial \bar{z}}d\bar{z}), $$
a tangent vector to $X_1$ can be written as 
$c_1\frac{\partial }{\partial x}+c_2\frac{\partial }{\partial y}+c_3\frac{\partial }{\partial z}+
\bar{c}_1\frac{\partial }{\partial \bar{x}}+\bar{c}_2\frac{\partial }{\partial \bar{y}}+\bar{c}_3\frac{\partial }{\partial \bar{z}}$, 
where the complex numbers $c_1, c_2, c_3$ satisfy 
\begin{eqnarray*} 
c_1\frac{\partial h}{\partial x}+c_2\frac{\partial h}{\partial y}+c_3\frac{\partial h}{\partial z}+\bar{c_1}\frac{\partial h}{\partial \bar{x}}+\bar{c_2}\frac{\partial h}{\partial \bar{y}}+\bar{c_3}\frac{\partial h}{\partial \bar{z}}=0 . 
\end{eqnarray*}
Taking the interior product with the standard symplectic structure $\omega _0$, we obtain the real $1$-form 
$$\frac{i}{2}(c_1d\bar{x}+c_2d\bar{y}+c_3d\bar{z}-\bar{c}_1dx-\bar{c}_2dy-\bar{c}_3dz).$$
Then the two vector fields $\mathrm{Re} w=\frac{1}{2}(w+\bar{w})$ and 
$\mathrm{Im} w=\frac{1}{2i}(w-\bar{w})$ form a basis of the bundle $(TX_1)^{\perp}$, 
where the vector $w$ is given by 
$$w=-\frac{\partial h}{\partial \bar{x}}\frac{\partial }{\partial x}
-\frac{\partial h}{\partial \bar{y}}\frac{\partial }{\partial y}
-\frac{\partial h}{\partial \bar{z}}\frac{\partial }{\partial z}
+\frac{\partial h}{\partial x}\frac{\partial }{\partial \bar{x}}
+\frac{\partial h}{\partial y}\frac{\partial }{\partial \bar{y}}
+\frac{\partial h}{\partial z}\frac{\partial }{\partial \bar{z}}$$
and $\bar{w}$ is its complex conjugate. 
Since $dh(w)=0$ and $dh(\bar{w})=\| \nabla h \|^2- \|\overline{\nabla} h \|^2$, 
we obtain 
$$dh(\mathrm{Re} w)=\frac{1}{2}(\| \nabla h \|^2- \|\overline{\nabla} h \|^2), \; 
dh(\mathrm{Im} w)=\frac{i}{2}(\| \nabla h \|^2- \|\overline{\nabla} h \|^2).$$
On the other hand, $dh(v')=dh(\nabla \rho^2)$ is calculated as 
\begin{eqnarray*}
dh(\nabla \rho^2)
=dh(E_x+E_y+E_z)=p\vphi_1x^p+q\vphi_2y^q+\vphi_3z^r+3axyz\\
=(p-3)\vphi_1x^p+(q-3)\vphi_2y^q+(r-3)\vphi_3z^r+\frac{3}{a}e^{i\theta}.
\end{eqnarray*}
Hence, it follows that $|dh(v')|<M$. Therefore, $|dg(v')|$ is estimated as follows:
\begin{eqnarray*}
|dg(v')|&< &\frac{2M}{\| \nabla h \|^2- \|\overline{\nabla} h \|^2}\max \left\{ |dg(w)|, |dg(\bar{w})| \right\}\\
&\leq & \frac{2M}{\| \nabla h \|^2- \|\overline{\nabla} h \|^2}\sqrt{|x|^2+|y|^2+|z|^2}(\| \nabla h \|+ \|\overline{\nabla} h \|)
\\
&\leq &\frac{2M}{\| \nabla h \|- \|\overline{\nabla} h \|}.
\end{eqnarray*}
As in the proof of Theorem~\ref{thm: main}, we obtain 
$$\| \nabla h \|- \|\overline{\nabla} h \|>\frac{a}{6}(\max\left\{ |x|, |y|, |z|\right\})^2-(M+6)\max\left\{ |x|, |y|, |z|\right\}.$$ 
By the condition $g(x,y,z)=|x|^2+e^{\frac{2\pi i}{3}} |y|^2+e^{\frac{4\pi i}{3}} |z|^2=\frac{1}{3}$, 
we have $$\max\left\{ |x|, |y|, |z|\right\}>\frac{1}{2}$$ on $\partial Y$. 
Then it follows that $$|dg(v')|<\frac{48M}{a-24M-144},$$ 
which is smaller than $\frac{2}{3}$ since $a$ is greater than $m^2(m+3)=2700M^2(10M+1)$. 
Therefore, the radial component of the vector field $dg(v)$ is positive along $\partial Y$, 
and hence, the Liouville vector field $v$ is transversal to $\partial Y$. 
This implies that $\partial Y$ is convex, and $X_1$, $Y$ and $Z$ can be connected by a homotopy of Liouville manifolds. 
\end{proof}

Thus we have a deformed Milnor fiber $Y=h^{-1}(\frac{1}{a}e^{i\theta })\cap g^{-1}(D^2_{\frac{1}{3}})\cap D^6$ 
together with a genus-one Lefschetz fibration $g|_Y$ that fibers the convex boundary $\partial Y$ by regular tori. 
Now we rewrite $Y$ by $Y_{\theta }$ in order to make it clear 
that we have a smooth family of deformed Milnor fibers parametrized by $\theta \in S^1$. 
Since $g|_{Y_{\theta }}\colon Y_{\theta }\to D^2_{\frac{1}{3}}$ is the genus-one Lefschetz fibration obtained in Theorem~\ref{precise statement}, we have an $S^1$-parametric Lefschetz fibration 
$$(g, h)\colon \bigcup _{\theta \in S^1} Y_{\theta }\to D^2_{\frac{1}{3}}\times S^1_{\frac{1}{a}},$$
where $h \colon \bigcup _{\theta \in S^1} Y_{\theta }\to S^1_{\frac{1}{a}}$ 
is isomorphic to the Milnor fibration $\arg f \colon S^5\setminus L\to S^1$as a fiber bundle over the circle.

\section{\large Milnor lattice and monodromy of the singularities}
\label{Section:MilnorLattice}
Let $X_{p,q,r}$ denote the Milnor fiber of a $T_{p,q,r}$ singularity. 
In the precedent section, we have deformed  $X_{p,q,r}$ to the 
total space $Y\subset X_1$ 
of the Lagrangian Lefschetz 
fibration $\displaystyle g|_{Y}: Y\to D^2_{\frac{1}{3}}$ 
by a convex Liouville homotopy.  
Thus, we can say that $X_{p,q,r}$ itself 
carries a Lefschetz fibration to the disk $D^2$. 
In this section, we construct a system of embedded surfaces 
representing a generator of $H_2(X_{p,q,r}; \Z)$ 
in the guide of the fibration $g|_{Y}$. 
Then we observe that the intersection matrix in this system  
coincides with the famous one in algebraic geometry. 
Consequently, we will see that the surface system in the fibration  
is a geometric realization of the Milnor lattice. 
We also describe the monodromy of the Milnor fibration. 

\subsection{A surface system realizing the Milnor lattice}\label{Subsec:Intersection}
We fix the parameter $\theta=0$. Then we see from 
Lemma~\ref{Hessian} and its proof that the singular fiber 
\[
\Sigma_1:=(g|_{Y})^{-1}(|a|^{-\frac{2}{p}})=\left\{|y|=|z|,
\,\, |x|^2-|y|^2=a^{-\frac{2}{p}}\right\}\cap Y
\] 
is the union of the smooth $2$-spheres 
\[
\Sigma_{1,j}:=\left\{\arg x\in \left[\frac{2\pi (j-1)}{p}, \frac{2\pi j}{p}\right]/2\pi \Z
\subset \R/2\pi\Z \right\}\cap \Sigma_1 
\]
for $j=0,\dots p-1$, and the other singular fibers are the similar unions
\[
\Sigma_2:=(g|_{Y})^{-1}(|a|^{-\frac{2}{q}}e^{\frac{2\pi i}{3}}
)=\bigcup_{k=0}^{q-1}\Sigma_{2,k},\quad
\Sigma_3:=(g|_{Y})^{-1}(|a|^{-\frac{2}{r}}e^{\frac{4\pi i}{3}})=\bigcup_{l=0}^{r-1}\Sigma_{3,l}. 
\]
These spheres are oriented after the regular fiber 
\[
T^2:=(g|_Y)^{-1}(0)=\left\{|x|=|y|=|z|=a^{-\frac{2}{3}},
\,\, \arg x+\arg y+\arg z=0 \right\}. 
\] 
Note that the symplectic structure of $X_{p,q,r}$ 
and the orientation of the base space $D^2$ determines the 
orientation of the fiber $T^2$ with respect to which the area form
\[
(-d\arg x \wedge d\arg y)|_{T^2}=(-d\arg y\wedge d\arg z)|_{T^2}=(-d\arg z\wedge d\arg x)|_{T^2} 
\] 
is positive. Then we have 
\begin{eqnarray*}
H_2(X_{p,q,r};\Z)=H_2(Y;\Z)\ni\, [T^2]&=&[\Sigma_{1,0}]+\cdots+[\Sigma_{1,p-1}]\\
&=&[\Sigma_{2,0}]+\cdots+[\Sigma_{2,q-1}]\\
&=&[\Sigma_{3,0}]+\cdots+[\Sigma_{3,r-1}],
\end{eqnarray*}
\begin{eqnarray*}
&&[\Sigma_{1,j}]\cdot[\Sigma_{2,k}]=[\Sigma_{2,k}]\cdot[\Sigma_{3,l}]
=[\Sigma_{3,l}]\cdot[\Sigma_{1,k}]=0,\\
&&[\Sigma_{1,j}]\cdot[\Sigma_{1,{j+1}}]=[\Sigma_{2,k}]\cdot[\Sigma_{2,k+1}]
=[\Sigma_{3,l}]\cdot[\Sigma_{3,l+1}]=1
\end{eqnarray*}
for any $j\in \Z_p$, $k\in \Z_q$, $l\in \Z_r$ provided that $p,q,r>2$. 
In the case where $p=2$, we have 
$[\Sigma_{1,0}]\cdot[\Sigma_{1,1}]=2$ since 
the intersection $\Sigma_{1,0}\cap \Sigma_{1,1}$ consists of two points. 
In any case, the displaceability $[T^2]\cdot[T^2]=0$ of the regular 
fiber implies 
\[
[\Sigma_{1,j}]\cdot[\Sigma_{1,j}]=[\Sigma_{2,k}]\cdot[\Sigma_{2,k}]
=[\Sigma_{3,l}]\cdot[\Sigma_{3,l}]=-2 \,\, 
(j\in \Z_p,\, k\in \Z_q,\, l\in \Z_r).\]
We take three oriented disks
\begin{align*}
&D_1:=\left\{ |y|=|z|, \,\, 0\leq |x|^2-|y|^2\leq a^{-\frac{2}{p}},\,\,
\arg x=0 \right\}\cap Y,\\
&D_2:=\left\{ |z|=|x|, \,\, 0\leq |y|^2-|z|^2\leq a^{-\frac{2}{q}},\,\, 
\arg y=0 \right\}\cap Y,\\
&D_3:=\left\{ |x|=|y|, \,\, 0\leq |z|^2-|x|^2\leq a^{-\frac{2}{r}},\,\, 
\arg z=0 \right\}\cap Y
\end{align*}
with polar coordinates 
\begin{align*}
&\rho_1:=\sqrt{1-a^{\frac{2}{p}}(|x|^2-|y|^2)}|_{D_1},
&&\psi_1:=(\arg(y)-\arg(z))|_{D_1},\\
&\rho_2:=\sqrt{1-a^{\frac{2}{q}}(|y|^2-|z|^2)}|_{D_2},
&&\psi_2:=(\arg(z)-\arg(x))|_{D_2},\\
&\rho_3:=\sqrt{1-a^{\frac{2}{r}}(|z|^2-|x|^2)}|_{D_3},
&&\psi_3:=(\arg(x)-\arg(y))|_{D_3}. 
\end{align*}
For each $m\in\{1,2,3\}$, 
the disk $D_m$ intersects 
with the sphere $\Sigma_{m,1}$ positively at the center of $D_m$. It also intersects 
with the sphere $\Sigma_{m,0}$ negatively at the same point. 
The image $g(D_m)$ is the line segment joining 
the singular value $g(\Sigma_m)$ to the origin. 
Over the origin, the boundary curves $\partial D_m$ ($m=1,2,3$) 
meet at a single point satisfying $\arg x=\arg y=\arg z=0$. 
Thus, the relative complement $T^2\setminus \bigcup_{m=1}^3\partial D_m$ 
is the union of the positive triangular 
region $T_+$ with the center satisfying 
$\arg x|_{T^2}=\arg y|_{T^2}=\arg z|_{T^2}=\frac{4\pi}{3}$ and 
the negative triangular region $T_-$ 
with the center satisfying 
$\arg x|_{T^2}=\arg y|_{T^2}=\arg z|_{T^2}=\frac{2\pi}{3}$. 
Here we orient $T_+$ positively and $T_-$ negatively 
with respect to the orientation of $T^2$ so that 
they are positive as well as $D_m$ ($m=1,2,3$) in each of 
the piecewise smooth spheres $\Sigma_\pm:=D_1\cup D_2\cup D_3\cup T_\pm$. 
Then we have   
\[
[\Sigma_+]-[\Sigma_-]=[T^2],\quad 
[\Sigma_\pm]\cdot[\Sigma_{m,1}]=1,\quad [\Sigma_\pm]\cdot [\Sigma_{m,0}]=-1\quad (m=1,2,3).
\] 
Now we obtain the system $\mathcal{S}$ of $p+q+r-1$ spheres 
\[
\Sigma_{1,j}\,\,(j=1,\dots, p-1),\,\, 
\Sigma_{2,k}\,\,(k=1,\dots, q-1),\,\,
\Sigma_{3,l}\,\,(l=1,\dots, r-1),\,\,
\Sigma_\pm 
\]
which represents a generator of $H_2(Y;\Z)=H_2(X_{p,q,r};\Z)$ as is shown in the next proposition. 
Note that, then, the system $\mathcal{S}'$ of $p+q+r-1$ surfaces 
\[
\Sigma_{1,j}\,\,(j=1,\dots, p-1),\,\, 
\Sigma_{2,k}\,\,(k=1,\dots, q-1),\,\,
\Sigma_{3,l}\,\,(l=1,\dots, r-1),\,\,
\Sigma_+, T^2
\]
represents another generator of $H_2(X_{p,q,r};\Z)$ (see Fig.~\ref{jyuzu2}). 
\begin{figure}[h]
\centering
\includegraphics[width=120mm]{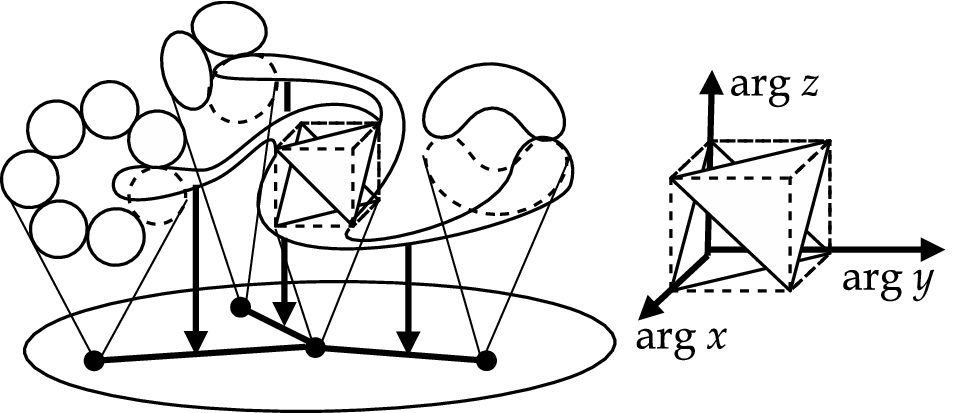}
\caption{The left-hand shows the $p+q+r-1$ spheres. The singular fibers are the unions of $p(=2)$, $q(=3)$, and $r(=7)$ spheres which look like rosaries. The dotted spheres $\Sigma_{m,0}$ ($m=1,2,3$) are removed.  Instead, the three disks $D_m$ over the bolded segments tie the broken rosaries to the regular fiber $T^2$ at the triangular region $T_+$ (or $T_-$). The right-hand shows how to cut $T_\pm$ out of $T^2$. Here the front triangle is the positive region $T_+$ which are seen from the back.}
\label{jyuzu2}
\end{figure}

\begin{proposition}
The homology classes of the spheres in $\mathcal{S}$ generates $H_2(X_{p,q,r};\Z)$.
\end{proposition}

\begin{proof}
To calculate the homology group, we deform the fibration $g|_Y$ 
to the generic one depicted in the right-hand of Fig.~\ref{Y12}. 

\begin{figure}[h]
\includegraphics[width=65mm]{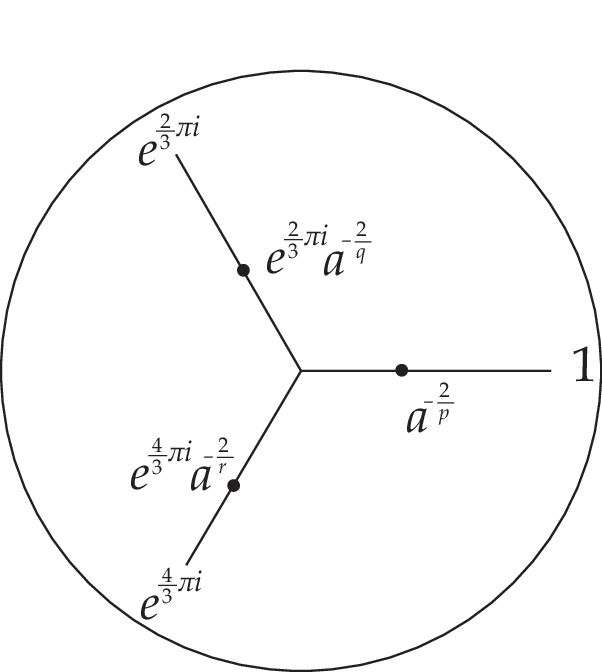}
\quad$\to$\quad
\includegraphics[width=65mm]{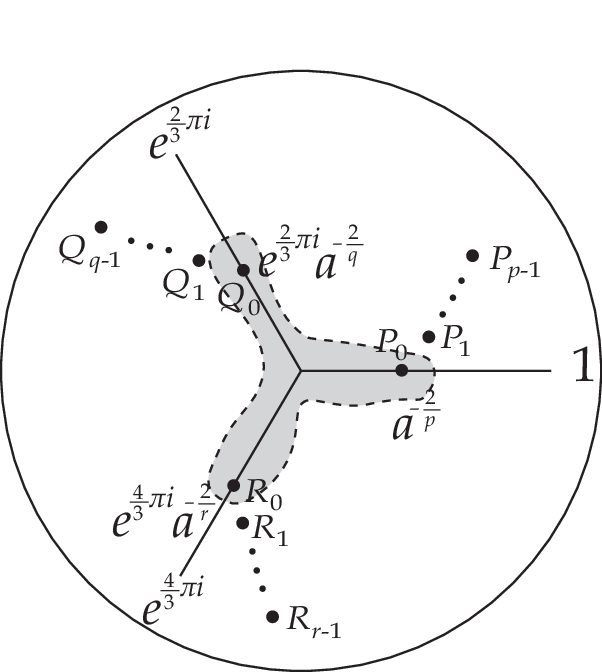}
\caption{The critical values of $g|_Y$ in the left-hand go into pieces in the right-hand. 
We suppose that the critical points over $P_j$, $Q_k$, $R_l$ satisfy 
$\arg x=\frac{2\pi j}{p}$, $\arg y=\frac{2\pi k}{q}$, $\arg z=\frac{2\pi l}{r}$ ($j\in \Z_p$, $k\in\Z_q$, $l\in \Z_r$).}
\label{Y12}
\end{figure}

It is easy to see that the preimage of the shadowed part is simply connected and 
the second homology group is $\Z \oplus \Z$ generated by $[\Sigma_+]$ and $[\Sigma_-]$. 
Then, using the Mayer-Vietoris exact sequence, 
we can extend the shadowed part so that it includes 
one more critical value, and successively we can calculate 
the homology group of its preimage. Note that, in the homology calculation, 
adding a Lefschetz singularity is indeed equivalent to attaching a disk along its vanishing cycle. 
This increases the number of the spheres $\Sigma_{*,*}$ by one. We leave the detail of the calculation to the readers.
\end{proof}

We have determined the 
intersection matrix with respect to the above generator 
of $H_2(X_{p,q,r};\R)$ except the four entries 
\[
[\Sigma_+]\cdot [\Sigma_+], \quad[\Sigma_+]\cdot[\Sigma_-],\quad
[\Sigma_-]\cdot [\Sigma_+], \quad[\Sigma_-]\cdot[\Sigma_-]. 
\]
From $[\Sigma_+]-[\Sigma_-]=[T^2]$ and $[T^2]\cdot [T^2]=0$, we 
see that the four entries mutually coincide. 
We will show that they are equal to $-2$. 
\begin{proposition}
$[\Sigma_+]\cdot [\Sigma_+]=-2$. 
\end{proposition}
\begin{proof}
We slightly deform the sphere $\Sigma_+$ into $\Sigma'_+$ so that the 
intersection $\Sigma_+\cap \Sigma'_+$ consists of four points 
at which smooth portions of the spheres meet transversely. 
Indeed, we deform it by making the image $g(\Sigma'_+)$ the union of 
the dotted arcs depicted in Fig.~\ref{Y3}. 

\begin{figure}[h]
\centering
\includegraphics[width=40mm]{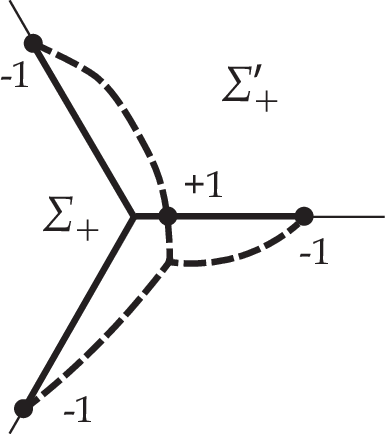}
\caption{The image of the sphere $\Sigma_+$ is changed from the union of the bold line segments 
to that of the dotted arcs. Each of the numbers indicates the sign of the intersection over there. }
\label{Y3}
\end{figure}

The intersection $\Sigma_+\cap \Sigma'_+$ consists 
of three negative intersections at the centers of the disks $D_m\subset \Sigma_+$ ($m=1,2,3$) 
and another point $(x,y,z)=(b,c,c)$ satisfying that $b,c\in\R_{>0}$, $bc^2=a^{-2}$ and $0<b<2\sqrt{2}c$. 
Then, we see that $(-2E_x+E_y+E_z, e_y-e_z)_{(b,c,c)}$ is an oriented basis of 
the tangent space $T_{(b,c,c)}\Sigma_+$.
Further, we may assume that $(-E_y+E_z, e_z-e_x)_{(b,c,c)}$ is an oriented basis of $T_{(b,c,c)}\Sigma'_+$. 
This implies that the intersection of $\Sigma_+$ and $\Sigma'_+$ at $(b,c,c)$ is positive. 
Therefore, $[\Sigma_+,\Sigma_+]=3\cdot(-1)+1=-2$. 
\end{proof}
We see that the above union of surfaces is a geometric realization 
of the left-hand diagram of the Milnor lattice in Fig.~\ref{TT(pqr)}.
 
\begin{figure}[h]
\includegraphics[width=55mm]{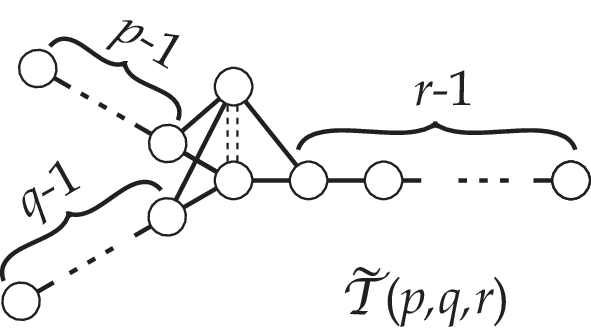}
\qquad\qquad
\includegraphics[width=55mm]{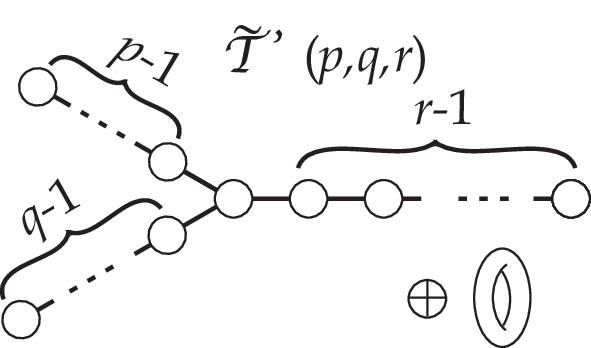}
\caption{The Milnor lattice. The left-hand diagram $\widetilde{\mathcal T}(p,q,r)$ is equivalent to the right-hand one $\widetilde{\mathcal T}'(p,q,r)=
{\mathcal T}(p,q,r) \oplus \langle[T^{2}]\rangle$.}
\label{TT(pqr)}
\end{figure}

Using the generator $\mathcal{S}=([\Sigma_{1,*}], [\Sigma_{2,*}], [\Sigma_{3,*}], [\Sigma_\pm])$, we have the intersection matrix 
\[
\begin{pmatrix}
-A_{p-1}&&&\delta_{*,1}&\delta_{*,1}\\
&-A_{q-1}&&\delta_{*,1}&\delta_{*,1}\\
&&-A_{r-1}&\delta_{*,1}&\delta_{*,1}\\
\delta_{1,*}&\delta_{1,*}&\delta_{1,*}&-2&-2\\
\delta_{1,*}&\delta_{1,*}&\delta_{1,*}&-2&-2
\end{pmatrix},
\]
where $A_n$ denotes the block $\begin{pmatrix}
2&-1&\\
-1&2&\ddots\\
&\ddots&\ddots&-1\\
&&-1&2
\end{pmatrix}$, and 
$\delta_{*,1}$ etc. the arrangement of 
the Kronecker delta $\delta_{*,*}$. 
We can also realize the left-hand diagram 
by using the generator 
$\mathcal{S}'=([\Sigma_{1,*}], [\Sigma_{2,*}], [\Sigma_{3,*}], [\Sigma_+], [T^2])$. Then the matrix is  
\[
\begin{pmatrix}
-A_{p-1}&&&\delta_{*,1}\\
&-A_{q-1}&&\delta_{*,1}\\
&&-A_{r-1}&\delta_{*,1}\\
\delta_{1,*}&\delta_{1,*}&\delta_{1,*}&-2\\
&&&&0
\end{pmatrix}.
\]
In the case where $\frac{1}{p}+\frac{1}{q}+\frac{1}{r}<1$, the annihilator is nothing other than $\langle [T^2] \rangle$ since 
\begin{align*}
\mathrm{disc}\, \mathcal{T}(p,q,r)&=
\begin{vmatrix}
-A_{p-1}&&&\delta_{*,1}\\
&-A_{q-1}&&\delta_{*,1}\\
&&-A_{r-1}&\delta_{*,1}\\
\delta_{1,*}&\delta_{1,*}&\delta_{1,*}&-2
\end{vmatrix}\\
&=(-1)^{p+q+r-2}\cdot(qr+rp+pq-pqr)\neq 0.
\end{align*}
Here we can show the equivalence 
\[
|\mathrm{disc}\, \mathcal{T}(p,q,r)|=1\quad \Longleftrightarrow \quad 
\{p,q,r\}=\{2,3,7\}.
\] 
In the case where $(p,q,r)=(2,3,7)$, if we permute the generator $\mathcal{S'}$ to 
\[
\mathcal{S''}=([\Sigma_{2,2}], [\Sigma_{2,1}], [\Sigma_+], [\Sigma_{3,1}],\dots, [\Sigma_{3,6}], [\Sigma_{1,1}], [T^2]), 
\]
the non-degenerate part of the intersection matrix  
becomes the negative of the Cartan matrix $\begin{pmatrix}
A_9&-\delta_{*,3}\\
-\delta_{3,*} &2
\end{pmatrix}
$ for $E_{10}(\cong E_8 \oplus H)$.  

\subsection{The monodromy of the Milnor fibration}
Let $\widetilde{\mu}: X_{p,q,r}\to X_{p,q,r}$ denote the monodromy map 
of the Milnor fibration of the $T_{p,q,r}$ singularity. 
We describe the induced map 
$\widetilde{\mu}_*: H_2(X_{p,q,r};\Z) \to H_2(X_{p,q,r})$ by using 
the second generator $\mathcal{S}'$ in the previous section. 
To this aim, we write $Y=Y_\theta$ to specify the parameter $\theta$ as in the last paragraph of \S~\ref{construction}, and 
consider the monodromy map $\mu: Y_{2\pi} \to Y_0$ of 
the fibration 
\[
\textrm{pr}\colon \bigsqcup_{\theta\in \R/2\pi\Z} Y_\theta\to \R/2\pi\Z: Y_\theta \mapsto \theta
\]
instead of $\widetilde{\mu}$. 
We trivialize the cylinder 
$\displaystyle \bigsqcup_{0\leq \theta <2\pi} Y_\theta \cong Y\times[0,2\pi)$ to define $\mu$.

We take an open set $N_0$ on the fiber $Y_0$ such that
\[
N_0\supset 
\{\arg g=\frac{\pi}{3},\,\, \frac{\pi}{2}\,\, \mathrm{or} \,\, \frac{5\pi}{3}\}\cap Y_0
\]
and $N_0\cap\textrm{supp}(\vphi_m)=\emptyset$ ($m=1,2,3$). For any point $(x,y,z)$ of $N_0$, 
we have the point $(x,y,ze^{-i\theta})$ on each fiber $Y_\theta$. 
This defines the trivialization $N_0\times [0,2\pi)$ 
such that $\mu|_{N_0}$ is the identity. Then $\mu(T^2)=T^2$, and therefore $\mu_*([T^2])=[T^2]$.

On the other hand the singularities of $\Sigma_1=(g|_{Y_\theta})^{-1}(a^{-\frac{2}{p}})$ are the points 
\[(x,y,z)=(a^{-\frac{1}{p}}\exp(\frac{2\pi j+\theta}{p}i),\; 0,\; 0)\quad (j=0,1,\dots p-1).\] 
Thus $\mu$ sends the spheres $\Sigma_{1,j}$, $\Sigma_{2,k}$, $\Sigma_{3,l}$ respectively to $\Sigma_{1,j+1}$, $\Sigma_{2,k+1}$, $\Sigma_{3,l+1}$ ($j\in \Z_p$, $k\in \Z_q$, $l\in \Z_r$). This implies    
\begin{align*}
&\mu_*([\Sigma_{1,j}])=[\Sigma_{1,j+1}]\,\,(j=1,\dots,p-2), & 
\mu_*([\Sigma_{1,p}])=[T^2]-[\Sigma_{1,1}]-\cdots-[\Sigma_{1,p}],\\
&\mu_*([\Sigma_{2,k}])=[\Sigma_{2,k+1}]\,\,(k=1,\dots,q-2), & 
\mu_*([\Sigma_{2,q}])=[T^2]-[\Sigma_{2,1}]-\cdots-[\Sigma_{2,q}],\\
&\mu_*([\Sigma_{3,l}])=[\Sigma_{3,l+1}]\,\,(l=1,\dots,r-2), & 
\mu_*([\Sigma_{3,r}])=[T^2]-[\Sigma_{3,1}]-\cdots-[\Sigma_{3,r}].
\end{align*}
Now we may assume that the monodromy map $\mu$ preserves the fibration $g|_{Y_0}$ and 
it is periodic near the singular fibers. Then we have 
\[
\mu_*([\Sigma_+])=[\Sigma_+]+[\Sigma_{1,1}]+[\Sigma_{2,1}]+[\Sigma_{3,1}]-[T^2], 
\]
where the last term $-[T^2]$ comes from the rotation of the last component 
of the point $(x,y,ze^{-i\theta})$ presenting $\{*\}\times [0,2\pi)\subset N_0\times [0,2\pi)$. 
In fact, the singularities of $\Sigma_3=(g|_{Y_\theta})^{-1}(e^{\frac{4\pi i}{3}} a^{-\frac{2}{r}})$ are the points 
\[(x,y,ze^{-i\theta})=(0,\; 0,\; a^{-\frac{1}{r}}
\exp(\frac{2\pi l-(r-1)\theta}{r}i))\quad (l=0,1,\dots r-1).\]  

\subsection{The section}
There is another remarkable surface properly embedded in the Milnor fiber $X_{p,q,r}$. 
It intersects with the regular fiber $T^2$ while it avoids any other surfaces in 
the second system $\mathcal{S}'$. Namely, 
\begin{proposition}\label{Prop:section}
The Lefschetz fibration of the Milnor fiber $X_{p,q,r}$ 
of a $T_{p,q,r}$ singularity  
$(\frac{1}{p}+\frac{1}{q}+\frac{1}{r}\leq 1)$ admits a section 
$s=s_{p,q,r} \colon D^{2} \to X_{p,q,r}$ 
which does not intersect with  
any of the spheres in the system $\mathcal{S}'$ representing the cycles of $\mathcal{T} (p,q,r)$.  
\end{proposition}
\begin{proof}
We suppose that the Milnor fiber is defined by $\theta=0$. 
Take a local section near the origin which intersects with $T^2$ at the center 
of the negative triangular region $T_-$. We suppose that it is expressed as 
\[
\arg x=\arg y=\arg z=\frac{2\pi}{3},\quad |xyz|=a^{-2}, 
\]
where $|x|$, $|y|$, $|z|$ are close to $a^{-\frac{2}{3}}$. 
On the other hand, there is a point on the negative region $T_-$ satisfying 
$-\frac{2\pi(p-1)}{p}<\arg x<2\pi$. 
On such a point, $\arg y$ and $\arg z$ are small positive angles. Then, recalling the definition 
\begin{align*}
&D_1:=\left\{ |y|=|z|, \,\, 0\leq |x|^2-|y|^2\leq a^{-\frac{2}{p}},\,\,
\arg x=0 \right\}\cap Y,\\
&D_2:=\left\{ |z|=|x|, \,\, 0\leq |y|^2-|z|^2\leq a^{-\frac{2}{q}},\,\, 
\arg y=0 \right\}\cap Y,\\
&D_3:=\left\{ |x|=|y|, \,\, 0\leq |z|^2-|x|^2\leq a^{-\frac{2}{r}},\,\, 
\arg z=0 \right\}\cap Y
\end{align*}
of the parts of $\Sigma_+$, 
we see that the section can be extended so that it does not intersect with $\Sigma_+$ and 
it intersect with each of $\Sigma_{m,0}\not\in \mathcal{S}'$ ($m=1,2,3$). 
\end{proof}

\section{\large Decomposition of K3 surface} 

In this section, we describe a smooth decomposition of a K3 surface 
into the two Milnor fibers of cusp singularities 
of a duality pair in the extended strange duality.  
For more precise definitions and explanations on strange and 
extended strange duality, 
refer to \cite{A}, \cite{P1}, \cite{EW}, or \cite{Nakamura2}.  
 
\subsection{The extended strange duality and Hirzebruch-Inoue surfaces}

\subsubsection{Strange duality of Arnol'd}
Among isolated hypersurface singularities of complex three variables, 
rational singularities, in other words, ADE-type singularities are 
exactly those of modality zero. 
Here the modality of a singularity 
is roughly the number of parameters involved in its classification. 
Simple-elliptic singularities and cusp singularities \textcolor{black}{contain the parameter $a$, and therefore they} are of modality 1 (also said to be unimodal, 1-modal, or 
unimodular).  Other than these, there still exist 14 unimodal singularities,  
which are called the {\it exceptional unimodal singularities}, \textcolor{black}{
listed in Table \ref{exceptional}}, 
and no more unimodal ones exist. 
\begin{table}[h]
\centering
\begin{tabular}{ccccc}\hline
Singularity  & Gabrielov \# & polynomial & Dolgachev \# & Dual 
\\
\hline
$E_{12}$ {\tiny a.k.a.}$S_{2, 3, 7}$
 & 2, 3, 7 & $x^{2}+y^{3} +z^{7}$ & 2, 3, 7 & $E_{12}$ 
\\
$Z_{11}$ \quad $S_{2, 4, 5}$
& 2, 4, 5 & $x^{2}+y^{3}z +z^{5}$ & 2, 3, 8 & $E_{13}$ 
\\
$Q_{10}$  \quad $S_{3, 3, 4}$
& 3, 3, 4 & $x^{3}+y^{2}z +z^{4}$ & 2, 3, 9 & $E_{14}$ 
\\
$E_{13}$  \quad $S_{2, 3, 8}$
& 2, 3, 8 & $x^{2}+y^{3} +yz^{5}$ & 2, 4, 5 & $Z_{11}$ 
\\
$Z_{12}$  \quad $S_{2, 4, 6}$
& 2, 4 6 & $x^{2}+y^{3}z +yz^{4}$ & 2, 4, 6 & $Z_{12}$ 
\\
$Q_{11}$  \quad $S_{3, 3, 5}$
& 3, 3, 5 & $x^{2}y+y^{3}z +z^{3}$ & 2, 4, 7 & $Z_{13}$ 
\\
$E_{14}$  \quad $S_{2, 3, 9}$
& 2, 3, 9 & $x^{3}+y^{2} +yz^{4}$ & 3, 3, 4 & $Q_{10}$ 
\\
$Z_{13}$  \quad $S_{2, 4, 7}$
& 2, 4, 7 & $x^{2}+xy^{3} +yz^{3}$ & 3, 3, 5 & $Q_{11}$ 
\\
$Q_{12}$  \quad $S_{3, 3, 6}$
& 3, 3, 6 & $x^{2}+y^{2}z +yz^{3}$ & 3, 3, 6 & $Q_{12}$ 
\\
$W_{12}$  \quad $S_{2, 5, 5}$
& 2, 5, 5 & $x^{5}+y^{2} +yz^{2}$ & 2, 5, 5 & $W_{12}$ 
\\
$S_{11}$  \quad $S_{3, 4, 4}$
& 3, 4, 4 & $x^{2}y+y^{2}z +z^{4}$ & 2, 5, 6 & $W_{13}$ 
\\
$W_{13}$  \quad $S_{2, 5, 6}$
& 2, 5, 6 & $x^{2}+xy^{2} +z^{4}$ & 3, 4, 4 & $S_{11}$ 
\\
$S_{12}$  \quad $S_{3, 4, 5}$
& 3, 4, 5 & $x^{3}y+y^{2}z +xz^{2}$ & 3, 4, 5 & $S_{12}$ 
\\
$U_{12}$  \quad $S_{4, 4, 4}$
& 4, 4, 4 & $x^{4}+y^{2}z +yz^{2}$ & 4, 4, 4 & $U_{12}$ 
\\
\hline
\end {tabular}
\vspace{\baselineskip}
\caption{14 exceptional unimodal singularities}
\label{exceptional}
\end{table}
\textcolor{black}{
In the table, the parameters are taken so that the 
polynomials defining the singularities are quasi-homogeneous. 
We notice that, for each of the singularities}, we have two labeling of 
\textcolor{black}{it} such as $E_{12}$  
and $S_{2,3,7}$. The second one is indexed by the Gabrielov triple  
which indicates the intersection form ${\mathcal T}(p,q,r)\oplus H$  
of the Milnor fiber, where 
$H=\begin{pmatrix}
0&1\\1&0
\end{pmatrix}$.   
In the Dynkin diagram ${\mathcal T}(p,q,r)$ 
each white vertex indicates a $-2$-rational curve and an edge 
connecting two vertices implies their positive transverse intersection (see Fig.~\ref{T(pqr)}). 

\begin{figure}[h]
\centering
\includegraphics[width=60mm]{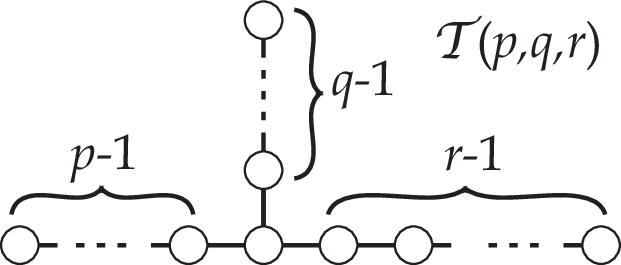}
\caption{The Dynkin diagram $\mathcal{T}(p,q,r)$}
\label{T(pqr)}
\end{figure}

The Dolgachev triple $(p', q', r')$ indicates a non-minimal resolution 
which consists of three mutually disjoint rational curves 
with self-intersection $-p'$, $-q'$, and $-r'$ and 
an exceptional rational curve transversely intersects 
with each of three rational curves (see Fig.~\ref{Dol}). 

\begin{figure}[h]
\centering
\includegraphics[width=110mm]{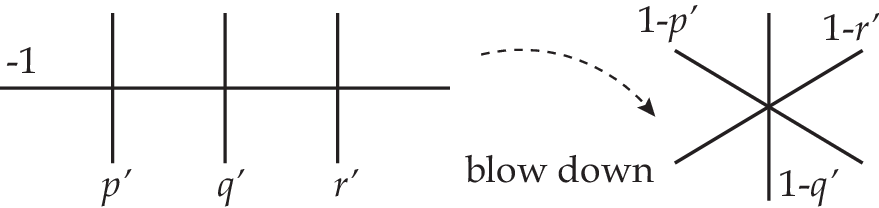}
\caption{The non-minimal resolution.}
\label{Dol}
\end{figure}

{\it Arnol'd's strange duality} consists of ten pairs of exceptional singularities, 
among which six pairs are self-dual and four are not. 
Many more interesting properties are exchanged between the dual 
partners, while the most remarkable phenomenon is that 
between the dual partners, the Gabrielov triple and the Dolgachev 
triple are exchanged.  

\subsubsection{Pinkham's interpretation by the K3 lattice}
\label{subsubsection:Pinkham}

Pinkham \cite{P1} interpreted these strange phenomena 
in the following way. 
The affine surface with the singularity $S_{p,q,r}$ admits a compactification 
whose resulting surface is smooth away from the original singularity 
and has the divisor at infinity which realizes the Dynkin diagram 
${\mathcal T}(p',q',r')$ of the dual partner, 
in other words, the Dynkin diagram indexed 
not by its Gabrielov triple but by its Dolgachev triple. 
Then the surface has a deformation to a K3 surface 
where it preserves the divisor at infinity while the singularity 
is smoothen so that the complement of the divisor is the Milnor fiber 
because of the quasi-homogeneity. 
Then it is seen that in the K3 lattice 
(the 2nd integral homology of a K3 surface with the intersection 
$2E_{8}\oplus 3H$, the rank is 22) the lattice 
${\mathcal T}(p',q',r')$ for the divisor at infinity  and 
that  ${\mathcal T}(p,q,r)\oplus H$  for the Milnor fiber 
are placed as the orthogonal complement. 

\subsubsection{The extended strange duality}

Nakamura \cite{Nakamura1} and Looijenga \cite{Lo} 
found that yet another but similar duality phenomenon 
exists among 14 cusp singularities $T_{p,q,r}$'s 
with exactly the same index triples as those 
of the exceptional singularities 
$S_{p,q,r}$'s.  The new duality is called 
the {\it  extended strange duality}. 

For the cusp singularity  $T_{p,q,r}$, 
the Milnor lattice 
is indicated by the Dynkin diagram $\widetilde{\mathcal T}(p,q,r)$ 
(see\cite{Gabrielov}, also \cite{Ke} and 
\S \ref{Section:MilnorLattice} of the present article).   
%
%
Similarly in the case of the strange duality of Arnol'd, 
between the cusp singularities in an extended strange duality pair, 
the structure of the Milnor lattice and that of 
the cycles of their resolutions are exchanged.  
If the triples $(p, q, r)$ and $(p', q',r')$ are dual to each other
in the list,  
\EG  $(2,3,9)$ and $(3,3,4)$, the cusp singularity  $T_{p,q,r}$ 
admits a resolution consisting of a cycle of three rational curves 
with self-intersection $1-p'$, $1-q'$, $1-r'$ 
each of which transversely intersects once to any others 
at distinct points. Remark here that if $p'=2$ (
 $p'\leq q' \leq r'$) the resolution is not minimal. 
 If $p'=2$ and $q'\geq 4$ then the first rational curve is exceptional and 
 blown down so that we have the minimal resolution consisting of 
 two rational curves  with self-intersection $2-q'$ and $2-r'$ 
 which transversely intersects to each other twice.  
 If $p'=2$ and $q'\geq 4$ after the first curve is blown down, 
 the second one becomes exceptional and is also blown down, 
so that the minimal resolution consists of a single rational curve 
with a node and self-intersection $6-r'$. See Fig.~\ref{cycles}. 

\begin{figure}[h]
\centering
\includegraphics[width=140mm]{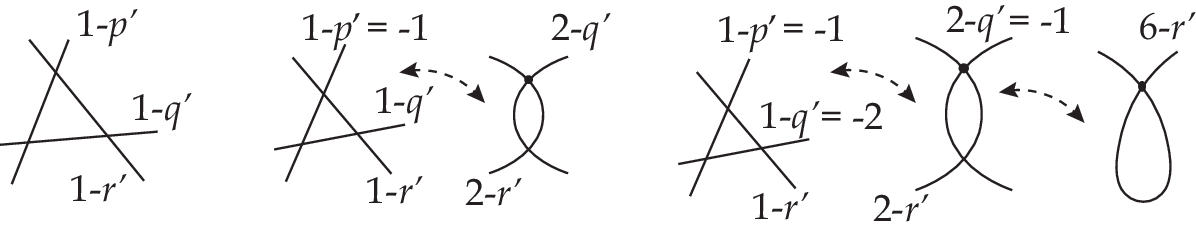}
\\
\smallskip
$3\leqq p'$
\qquad\qquad
$p'=2,\quad 4\leqq q' $
\qquad\qquad\qquad\qquad
$p'=2, q'=3$
\qquad
\qquad
\qquad

\caption{The cycles of rational curves.}
\label{cycles}
\end{figure}

Now recall that the link of the cusp singularity $T_{p,q,r}$ is the $T^{2}$-bundle 
over the circle with the hyperbolic monodromy  
$$
A_{p,q,r}
=\begin{pmatrix}   
r-1&-1\\1&0
\end{pmatrix}
\begin{pmatrix}   
q-1&-1\\1&0
\end{pmatrix}
\begin{pmatrix}   
p-1&-1\\1&0
\end{pmatrix}. 
$$ 
\begin{proposition}\label{Prop:DualMonodromies}
The triples $(p, q, r)$ and $(p', q',r')$ are dual to each other 
in the extended strange duality if and only if the monodromies 
$A_{p,q,r}$ and $A_{p',q',r'}$ are conjugate to the inverse of 
the other in ${\mathit{SL}}(2;\Z)$.   
\end{proposition}

\begin{proof}
Put $J=\begin{pmatrix}0&-1\\1&0\end{pmatrix}$, $T=\begin{pmatrix}0&1\\-1&-1\end{pmatrix}$ and take 
$\pm J$, $\pm T$ in $\mathit{PSL}(2,\Z)$ which generate subgroups isomorphic to $\Z_2$, $\Z_3$, respectively. Since $\mathit{PSL}(2,\Z)$ is the free product of these subgroups, 
any matrix in $\mathit{SL}(2,\Z)$ is conjugate ($\sim$) in $\mathit{SL}(2,\Z)$ to 
\[
(-1)^s JT^{a_1}JT^{a_2}\cdots JT^{a_m}(\sim (-1)^sJT^{a_2}\cdots JT^{a_m}JT^{a_1}\sim\cdots)
\]
where $s\in\{0,1\}$ and a cycle of $a_k\in \{1,2\}$ ($k\in \Z_m$) are uniuely determined. 
Let the signed cycle $(-1)^s[a_1,\dots, a_m]$ denote the conjugacy class. The inverse class is 
$(-1)^{s+m}[3-a_m,\dots, 3-a_2,3-a_1]$. As for $A_{p,q,r}=J(TJ)^{r-1} J(TJ)^{q-1} J(TJ)^{p-1}$, 
\begin{align*}
A_{2,3,r}&= JTJTJTJ(TJ)^{r-6}TJTJ(JTJTJ)(JTJ)
\\
&
\quad \sim -J(TJ)^{r-6}T^2\in -[1^{r-6}, 2]\quad (r\geq 6)
\\
A_{2,q,r}&= JTJTJ(TJ)^{r-4}TJJTJ(TJ)^{q-4}TJTJ(JTJ)
\\
&\quad \sim J(TJ)^{r-4}T^2J(TJ)^{q-4}T^2\in [1^{r-4},2,1^{q-4},2]\quad(r\geq q\geq 4)
\\
A_{p,q,r}&= JTJ(TJ)^{r-3}TJJTJ(TJ)^{q-3}TJJTJ(TJ)^{p-3}TJ
\\
&
\quad\in -[1^{r-3},2,1^{q-3},2,1^{p-3},2]\quad (r\geq q\geq p\geq 3)
\end{align*}
where each power denotes the iteration, e.g., $(1,2)^2=1,2,1,2$. Since the cardinal of $\{k\mid a_k=2\}$ is at least $1$ and at most $3$, that of $\{k\mid a_k=1\}$ must also be at least $1$ and at most $3$ for the duality. Thus it is enough to consider 
\begin{align*}
&A_{2,3,r}~(r=7,8,9),\quad A_{2,q,r}~((q,r)=(4,5),(4,6),(4,7),(5,5),(5,6))\\
& A_{p,q,r}~((p,q,r)=(3,3,4),(3,3,5),(3,3,6),(3,4,4),(3,4,5),(4,4,4)).
\end{align*}
This narrows down the possibilities to the listed 14 items. We have
\begin{align*}
&A_{2,3,7}, A_{2,3,7}^{-1}\in-[1,2],
&&A_{2,3,8}\in-[1^2,2],~~ A_{2,3,8}^{-1}\in[1,2^2],\\
&A_{2,3,9}\in-[1^3,2],~~ A_{2,3,9}^{-1}\in-[1,2^3],
&&A_{2,4,5}\in[1,2^2],~~ A_{2,4,5}^{-1}\in-[1^2,2],\\
&A_{2,4,6}, A_{2,4,6}^{-1}\in[1^2,2^2],
&&A_{2,4,7}\in[1^3,2^2],~~ A_{2,4,7}^{-1}\in-[1^2,2^3],\\
&A_{2,5,5}, A_{2,5,5}^{-1}\in[(1,2)^2],
&&A_{2,5,6}\in[1^2,2,1,2],~~ A_{2,5,6}^{-1}\in-[1,2,1,2^2],\\
&A_{3,3,4}\in-[1,2^3],~~ A_{3,3,4}^{-1}\in-[1^3,2],
&&A_{3,3,5}\in-[1^2,2^3],~~ A_{3,3,5}^{-1}\in[1^3,2^2],\\
&A_{3,3,6}, A_{3,3,6}^{-1}\in-[1^3,2^3],
&&A_{3,4,4}\in-[1,2,1,2^2],~~ A_{3,4,4}^{-1}\in[1^2,2,1,2],\\
&A_{3,4,5}, A_{3,4,5}^{-1}\in-[1^2,2,1,2^2],
&&A_{4,4,4}, A_{4,4,4}^{-1}\in-[(1,2)^3].
\end{align*}
\end{proof}

\begin{remark}\label{Cor:DualMonodromies}
For an extended strange duality pair of cusp singularities 
$T_{p,q,r}$ and $T_{p',q',r'}$, 
their links are isomorphic to each other as oriented $T^{2}$-bundles 
over the circle if the orientation of the base circle of one of two is reversed.  
Therefore the Milnor fibers of $T_{p,q,r}$ and $T_{p',q',r'}$ can be 
glued together along their boundary Sol-manifolds   
without changing the orientation as 4-manifolds. 
Further, since any gluing diffeomorphism is isotopic to one preserving the 
$T^{2}$-bundle structures, we obtain an oriented closed $4$-manifold 
equipped with a genus-one Lefschetz fibration, regardless of the choice of the gluing diffeomorphism. 
For details of the mapping class group of a $Sol$-manifold (hyperbolic torus bundle), see \cite{Wa, GM}. 
Note that we can composite the gluing diffeomorphism with $\begin{pmatrix}
-1&0\\0&-1\end{pmatrix}$ in the fiber direction, 
which does not change the resultant Lefchetz fibration. 
Further, except in the case of the self-dual pair $(p,q,r)=(p',q',r')=(3,4,5)$, we can reverse the orientation of the fiber and the base of the Lefschetz fibration of one of the Milnor fibers before the gluing 
by using $X=\begin{pmatrix}0&1\\1&0\end{pmatrix}$ according to the formulas 
\begin{align*}
&XA_{2,3,7}X\sim A_{2,3,7}^{-1}\in -[2,1](=-[1,2]),&
&XA_{2,3,8}X\sim A_{2,3,8}^{-1}\in[2^2,1],\\
&XA_{2,3,9}X\sim A_{2,3,9}^{-1}\in-[2^3,1],&
&XA_{2,4,5}X\sim A_{2,4,5}^{-1}\in-[2,1^2],\\
&XA_{2,4,6}X\sim A_{2,4,6}^{-1}\in[2^2,1^2],&
&XA_{2,4,7}X\sim A_{2,4,7}^{-1}\in-[2^3,1^2],\\
&XA_{2,5,5}X\sim A_{2,5,5}^{-1}\in[(2,1)^2],&
&XA_{2,5,6}X\sim A_{2,5,6}^{-1}\in-[2^2,1,2,1],\\
&XA_{3,3,4}X\sim A_{3,3,4}^{-1}\in -[2,1^3],&
&XA_{3,3,5}X\sim A_{3,3,5}^{-1}\in [2^2,1^3],\\
&XA_{3,3,6}X\sim A_{3,3,6}^{-1}\in-[2^3,1^3],&
&XA_{3,4,4}X\sim A_{3,4,4}^{-1}\in[2,1,2,1^2],\\
&XA_{3,4,5}X\in-[2^2,1,2,1^2]\not\ni A_{3,4,5}^{-1},&
&XA_{4,4,4}X\sim A_{4,4,4}^{-1}\in-[(2,1)^3]
\end{align*}
derived from $X^{-1}=X$, $XJX=-J$, and $XTX=T^2$.  
This does not change the diffeomorphism-type of the resultant closed manifold from Theorem \ref{Kas-Moishezon}, while it changes the overlapping pattern of the critical values of the Lefschetz fibration. This phenomenon can also be understood through the following observation: If we drop the assumption $p\leq q\leq r$ and consider the case where $q>r$, the calculations in the above proof become as follows:  
\begin{align*}
A_{2,q,3}&=(JTJTJ)JTJTJ(TJ)^{q-6}TJTJTJ(JTJ)\\
&\quad \sim -J(TJ)^{q-6}T^2 \sim A_{2,3,q},\\
A_{2,q,4}&\sim JT^2J(TJ)^{q-4}T^2\sim J(TJ)^{q-4}T^2JT^2\sim A_{2,4,q},
\\
A_{2,6,5}&\sim J(TJ)T^2J(TJ)^2T^2\sim J(TJ)^2T^2J(TJ)T^2\sim A_{2,5,6},
\\
A_{3,q,3}&\sim JT^2J(TJ)^{q-3}T^2JT^2
\sim -J(TJ)^{q-3}T^2JT^2JT^2\sim A_{3,3,q},
\\
A_{3,5,4}&\sim -J(TJ)T^2J(TJ)^2T^2JT^2\sim A^{-1}_{3,5,4}\\
&\quad \not\sim
-J(TJ)^2T^2J(TJ)T^2JT^2\sim A_{3,4,5}\sim A^{-1}_{3,4,5}. 
\end{align*}
Thus, except in the case where $(p,q,r)=(p',q',r')=(3,4,5)$, we can glue the Milnor fibers even after swapping the subscript (or the coordinates of $\C^3$) of one of the strange duality pair. 
It is worth mentioning that 
$A_{2,3,7}
=\begin{pmatrix}   
5&-11\\1&-2
\end{pmatrix}
$ 
is conjugate in ${\mathit{SL}}(2;\Z)$ to  
$
\begin{pmatrix}
2&1\\1&1
\end{pmatrix}
$, which is known as Arnold's cat map. 
\end{remark}

As in the above proof, the statement of Proposition~\ref{Prop:DualMonodromies} 
can be proven by a direct calculation for each individual case. 
However, this relationship between the duality of cusp singularities and the conjugacy classes of the corresponding monodromy matrices 
becomes more conceptually clear from the geometry of the Hirzebruch-Inoue surfaces discussed below. 

\subsubsection
{Hirzebruch-Inoue surfaces}
(\cite{H1, H2, HV, HZ, Inoue, Nakamura2})

Hirzebruch considered certain classes of complex surfaces 
in order to understand the resolutions of
cusp singularities.  
Inoue had a rather different motivation from the complex analysis, 
namely, looking for surfaces without meromorphic functions.   

Let $K$ be a real quadratic field and $'$ denote the conjugation. 
Take a complete module $\mathfrak{m}$ (\IE a free $\Z$ module of rank 2 in $K$), 
the positive multiplicative automorphism group 
$U^{+}(\mathfrak{m})=\{\alpha\in K\,;\, \alpha \mathfrak{m}=\mathfrak{m}, \alpha>0,\, \alpha'>0\}$ which is 
known to be infinite cyclic, and its subgroup 
$V=\langle \alpha_{V}\rangle$ ($\alpha_{V}>1$)
of finite index. 
Their natural semi-direct product $\Gamma=\Gamma(\mathfrak{m},V)$ 
acts on $\Hyp^{2}$ and on $\Hyp\times\C$ 
freely and discontinuously  
by $m\cdot (z_{1}, z_{2})=(z_{1}+m, z_{2}+m')$ and 
$\alpha \cdot (z_{1}, z_{2})=(\alpha z_{1}, \alpha'z_{2})$ for 
$m\in \mathfrak{m}$ and $\alpha\in V$,  
where 
$\Hyp$ (resp. $\Low $) denotes the upper (resp. lower) half plane in $\C $.  
Then we take the following quotients, which are non-singular surfaces; 
$$
X'(\mathfrak{m},V)=\Hyp^{2}/\Gamma, 
\quad
S'(\mathfrak{m},V)=\Hyp\times\C/\Gamma, 
\quad
\check{X}'(\mathfrak{m},V)=\Hyp\times\Low/\Gamma.
$$
By adding two points at infinity $\infty$ and $\infty^{-}$, 
the surface $S'(\mathfrak{m},V)$ is compactified 
so as to become a singular normal surface 
$S_{s}(\mathfrak{m},V)$, which is called a {\em singular Hirzebruch-Inoue surface}. 
Accordingly, we have 
$X_{s}(\mathfrak{m},V)=X'(\mathfrak{m},V)\cup\{\infty\}$ 
and 
$\check{X}_{s}(\mathfrak{m},V)=\check{X}'(\mathfrak{m},V)\cup\{\infty^{-}\}$ 
which have common boundary 
$\Hyp\times\R/\Gamma$ in $S_{s}(\mathfrak{m},V)$.   
The germ $(X_{s}(\mathfrak{m},V), \infty)$ at $\infty$ is called 
a {\em cusp singularity of type $(\mathfrak{m},V)$} and also 
a {\em Hilbert modular cusp}. 
It provides a model of a singularity of a 
{\em Hilbert modular surface}, which is the coarse moduli space for principally polarized abelian surfaces with a real multiplication structure.  

\subsubsection{Duality of Hilbert modular cusps.}

\textcolor{black}{The lower half 
$\check{X}_{s}(\mathfrak{m},V)$ is 
isomorphic to the Hilbert modular cusp} 
${X}_{s}(\mathfrak{m}',V')$ 
for some $\mathfrak{m}'$ and $V'$ 
(practically  $V=V', \alpha_{V}=\alpha_{V'}$) and the duality between 
 $(X_{s}(\mathfrak{m},V), \infty)$  and 
 $(\check{X}_{s}(\mathfrak{m},V), \infty^{-})$ 
$\cong$ 
$(X_{s}(\mathfrak{m}',V'), \infty)$ is a further extension 
of the extended strange duality. 

Since $\alpha\alpha'=1$ holds for $\alpha \in V$,  
the action of $\Gamma$ on $\Hyp \times \Low$ preserves 
the function $h=y_{1}y_{2}$.   In the coordinates $(h, y_{1}, x_{1}, x_{2})$ 
of $\Hyp \times \Low \cong \R\times\R_{+}\times\R^{2}$, 
where $z_{j}=x_{j}+iy_{j}$ ($i=1,2$), 
the action of $(m,\alpha)\in \Gamma$ is 
$(h, y_{1}, x_{1}, x_{2})\mapsto 
(h, \alpha y_{1}, \alpha(x_{1}+m), \alpha(x_{2}+m'))$.  
Therefore 
$S(\mathfrak{m},V)$ is diffeomorphic to 
the product of $\R$ and 
the 3-dimensional Sol-manifold 
$M_{\Gamma}=\R_{+}\times\R^{2}/\Gamma$, 
and $M_{\Gamma}$ is the common boundary of 
${X}_{s}(\mathfrak{m},V)$ and $\check{X}_{s}(\mathfrak{m},V)$ 
in $\check{S}_{s}(\mathfrak{m},V)$, while, 
the orientation of $M_{\Gamma}$ is reversed.  
This means as an oriented $T^{2}$ bundle over the circle, 
the orientation of the base circle is reversed, 
and consequently, the monodromy is changed into its inverse. 
This is an explanation of Proposition \ref{Prop:DualMonodromies} 
and Remark \ref{Cor:DualMonodromies}. 
%
%
%

Now we describe the correspondence between $(p,q,r)$'s and $(\m,V)$'s in terms of modified continued fractions.
 
The surface $S_{s}(\m,V)$ has two singularities at $\infty$ and $\infty^{-}$ 
which can be resolved by replacing with cycles $C$ and $D$ of 
rational curves \cite{H2}.  The cycle $C=C_{1}+ \cdots +C_{n}$ consists of 
$n$ rational curves so that for $n\geqq 3$, they intersect as 
$C_{j}^{2}=-c_{j}$ and $C_{j}C_{j+1}=1$ for $j=1, \ldots , n$ mod $n$  
($c_{j}\geqq 2$ for all $j$ and $c_{k}\geqq 3$ for some $k$)
and $C_{j}C_{k}=0$, otherwise, for $n=2$, $C_{1}$ and $C_{2}$ intersect 
transversely to each other at distinct two points as depicted in the middle 
of {\color{black} Fig.~\ref{cycles}}, and for $n=1$, $C_{1}$ is a rational 
curve with a node whose self-intersection $C_1^2$ is negative. 
Here we note that the self-intersection number differs from the normal Euler number by $2$ if $n=1$, 
so we should set $c_{1}=-C_{1}^{2}+2$ in this case. 
The resolution locus $D$ of $\infty^{-}$ is 
also a cycle of rational curves with the same property. 
We obtain a compact non-singular surface $S(\m,V)$ from $S_{s}(\m,V)$ by this resolution. 
This complex surface $S(\m,V)$ is called a {\em Hirzebruch-Inoue surface}. 

A cusp singularity $T_{p,q,r}$ ($1/p+1/q+1/r < 1$) 
appears as $(X_{s}(\mathfrak{m},V), \infty)$ for the case where 
the number of the cycle $D$ is at most $3$ and 
in fact every such $(p,q,r)$ appears. 
The 10 pairs of the extended strange duality exactly  
correspond to the surfaces $S(\mathfrak{m},V)$ for the case where 
both cycles $C$ and $D$ consist of less than or equal to three rational curves. 

In the general case of resolution cycle $C=C_{1}+\cdots +C_{k}$ with 
self-intersection $C_{j}^{2}=-c_{j}\leqq -2$ ($j=1,\cdots,k$) 
and $c_{i}\geqq 3$ for some $i$, and only in the case $k=1$ 
$c_{1}=-C_{1}^{2}+2$, 
by a cyclic permutation,  
it is assumed to be in the following form; 
$$
c_{1}, c_{2}, \cdots , c_{k}
= 
\gamma_{1}, \overbrace{2, \cdots, 2}^{\delta_{n}-3}, 
\gamma_{2}, \overbrace{2, \cdots, 2}^{\delta_{n-1}-3}, 
\cdots
\gamma_{n}, \overbrace{2, \cdots, 2}^{\delta_{1}-3} 
$$
where $\gamma_{1}, \cdots, \gamma_{n}\geqq 3$ 
and $\delta_{1}, \cdots, \delta_{n}\geqq 3$, 
so that $n$ is the number of $c_{j}$'s greater than or equal to 3. 
Then the dual cycle $D=D_{1}+\cdots+D_{l}$ with self-intersection 
$D_{j}^{2}=-d_{j}$ (the same remark as above applies for the case $l=1$) 
is given by the following rule (see \cite{HZ, Nakamura2}); 
$$
d_{l}, d_{l-1}, \cdots , d_{2}, d_{1}
= 
\overbrace{2, \cdots, 2}^{\gamma_{1}-3}, \delta_{n}, 
\overbrace{2, \cdots, 2}^{\gamma_{2}-3}, \delta_{n-1},
\cdots, 
\overbrace{2, \cdots, 2}^{\gamma_{n}-3}, 
\delta_{1}.$$
 The complete module $\mathfrak{m}=\Z \oplus \Z\omega_{C}$ 
for $C$ is given by 
$\omega_{C}=[\![\overline{c_{1}\cdots c_{{k}}}]\!]$ 
and  the automorphism group $V=\langle \alpha_{V} \rangle \subset U^{+}(\mathfrak{m})$ 
is generated by the product 
$\alpha_{V}=\prod_{j=1}^{k}\omega_{C^{(j)}}$, where
$[\![\overline{c_{1}\cdots c_{{k}}}]\!]$ denotes the repeating modified continued fraction 
\begin{eqnarray*}
c_1-\dfrac{1}{c_2-\dfrac{1}{c_3-\dfrac{1}{\ddots}}}
\end{eqnarray*}
for the number series $\{ c_j \}_{j\in \Z}$ of period $k$, 
and $C^{(j)}$'s ($j=1,\cdots,k$) denote all $k$ cyclic permutations of $C$. 
Here the ordering of the dual cycles and the corresponding quadratic irrationals 
are slightly modified from those in \cite{HZ, Nakamura2}. 

If we start from a $T_{p,q,r}$ singularity with $1/p+1/q+1/r<1$, 
and identify it with $(X(\mathfrak{m}, V),\infty)$, we first obtain 
the data of the resolution cycle $D=D_{1}+\cdots+D_{l}$ of 
 $(\check{X}(\mathfrak{m}, V),\infty)$ in the following way. 
If $p\geqq 3$ then we have $l=3$ and 
put $(d_{1}, d_{2}, d_{3})=(q-1,r-1,p-1)$,  
if $p=2$ and $q\geqq 4$, then $l=2$ and  $(d_{1}, d_{2})=(q-2,r-2)$, 
if $p=2$ and $q=3$, then $l=1$ and $d_{1}=r-4$.   
Then the resolution cycle $C$ for  $(X(\mathfrak{m}, V),\infty)$ 
is obtained by the above rule. 
Also note that $\alpha_{V'}=\overline{\alpha_{V^{-1}}}$.

\begin{example} \quad For $T_{2,3,8}$, the dual resolution cycle is computed as 
follows.  $(2,3,8)$ $\mapsto$ $(-1,-2,-7)$,  blown down to $(-1, -6)$, 
again blown down to $-d_{1}=-4=D_{1}^{2}-2$, 
so that we see $\omega_{D}=[\![\overline{4}]\!]=2+\sqrt 3=\alpha_{V'}$.  
From this $(c_{1},c_{2})$ is turned out to be $(3,2)$, and obtain 
 $\ds \omega_{C}=[\![\overline{32}]\!]=\frac{3+\sqrt 3}{2}$, 
$\ds [\![\overline{23}]\!]=\frac{3+\sqrt 3}{3}$, 
and $\ds \alpha_{V}=\frac{3+\sqrt 3}{2}\cdot\frac{3+\sqrt 3}{3}
=2+\sqrt 3=\alpha_{V'}$.  
$(-2, -3)$ is blown up to $(-1, -3, -4)$ 
and hence we see the dual triple 
$(p',q',r')=(2,4,5)$. 
\par
The actions (\IE multiplications) 
of $\alpha_{V}=\alpha_{V'}=2+\sqrt 3$ on  
$\m=\Z \oplus \Z\omega_{C}$ 
and on  $\m^{\star}=\Z \oplus \Z\omega_{D}$ 
are easily seen to be conjugate over 
$\mathit{SL}(2;\Z)$ to $A_{2,3,8}$ and $A_{2,4,5}$, respectively,  
and are conjugate not to each other but 
to the inverse of each other.   
\end{example}

\subsection{Smooth Decomposition of K3 Surface} \quad
Now we show that a K3-surface can be topologically decomposed into two Milnor fibers along an embedded Sol-manifold 
in ten distinct ways, corresponding to extended strange duality pairs of cusp singularities. 
Let $T_{p,q,r}$ and $T_{p',q',r'}$,   
be the cusp singularities of 
an extended strange duality pair  
and 
 $X_{1}$ and  $X_2$ their Milnor fibers,  
 respectively.  As shown in the Main Theorem,  
 $X_{1}$ and $X_{2}$ admit Lefschetz fibrations 
 $F_{1} : X_{1}\to D^{2}$ and $F_{2}:X_{2}\to D^{2}$ 
and Proposition \ref{Prop:DualMonodromies} implies that they 
are glued together so as to become a Lefschetz 
fibration 
$$
F = F_{1}\cup F_{2} : \hat{X}=X_{1}
\cup_{\partial } X_{2}\to S^{2}
$$
of a closed 4-manifold $\hat{X}$. 
A previous result of \cite{Mi2} tells that we can also deform 
the symplectic structures on 
 $X_{1}$ and $X_{2}$ to non-exact ones 
which smoothly coincide 
on the boundaries  and thus $\hat{X}$ is a closed 
symplectic $4$-manifold.  
Remark here that for this symplectic structure, 
the fiber tori are not Lagrangian but mostly symplectic. 
Let $W$ be an elliptic K3 surface, 
with a generic elliptic fibration $\Phi:W\to \CP^{1}$, 
namely,  
every critical point is simple, \IE of complex Morse type, 
and each singular fiber contains only one critical point. 

\begin{theorem}[Smooth Decomposition of K3 Surface]
\label{Thm:MainCorollary}  \quad
For any of the extended strange duality pair 
of cusp singularities, 
$\Phi:W\to \CP^{1}$ and $F:\hat{X}\to S^{2}$ are 
smoothly isomorphic as Lefschetz fibrations 
over the 2-sphere.  
\par
It can be paraphrased that 
the smoothing of 
the singular Hirzebruch-Inoue surface 
corresponding to an extended strange duality pair 
by replacing the neighborhoods of two singularities with their 
Milnor fibers is diffeomorphic to a K3 surface. 
\end{theorem}

\begin{proof}
Notice that for any of the $10$ pairs of dual triples $(p, q, r)$ and $(p', q', r')$, it holds that $$p+q+r+p'+q'+r'=24.$$ 
Hence, the Lefschetz fibration $F$ has exactly $24$ critical points by Theorem~\ref{precise statement}. 
As was mentioned in Remark~\ref{E2}, also $H$ has just $24$ critical points as a genus-one Lefschetz fibration. 
Then, by Theorem~\ref{Kas-Moishezon}, both $F$ and $H$ 
are isomorphic to $f_2\colon E(2)\to S^2$ as a Lefschetz fibration. 
\end{proof}
\begin{remark} 
Nakamura \cite{Nakamura2, Nakamura3} showed 
that a singular Hirzebruch-Inoue surface $S_s(\mathfrak{m}, V)$ admits 
a deformation to K3 surfaces if and only if 
the two singularities $\infty $ and $\infty ^{-}$ form an extended strange duality pair.    
This condition is also equivalent to requiring that 
both $\infty $ and $\infty ^{-}$ can be embedded in $\C ^3$ 
as singularities of type $T_{p,q,r}$ and $T_{p',q',r'}$, respectively, for some triples $(p,q,r)$ and $(p',q',r')$.   
Hence, our result can be interpreted as a topological analogue of this phenomenon, equipped with elliptic fibrations. 
\end {remark}
\begin{example} As $T_{2,3,7}$ is self-dual pair in the extended strange duality, 
a K3 surface is decomposed into 
a pair of the same Milnor fiber $X_{2,3,7}$. 
Among 10 extended strange duality pairs,  
$T_{2,3,7}$ - $T_{2,3,7}$ gives rise to a particularly nice 
decomposition as explained below. 

The intersection form of  $X_{2,3,7}$ is isomorphic to 
$\mathcal{T}(2,3,7)\oplus \langle [T^{2}] \rangle$ 
and from Proposition \ref{Prop:section}
the Lefschetz fibration $X_{2,3,7} \to D^2$ 
admits a section $s_{2,3,7}$ 
which does not intersect the cycles of $\mathcal{T}(2,3,7)$ 
but does once the regular fiber $[T^{2}]$. 
Of course, this applies to 
each of two copies. The boundary is a $T^{2}$-bundle with monodromy 
conjugate to
$
\begin{pmatrix}   
2&1\\1&1
\end{pmatrix}
$.  
Since ``this monodromy $-$ the identity'' is unimodular,  
we easily see the following claim. 
This holds only for this monodromy. 
\begin{assertion} 
Any two sections to the $T^{2}$-bundle 
over the circle with the above monodromy 
are homotopic as section to each other.  
\end{assertion}
\par 
This assertion enables us to obtain a section $S$ to the Lefschetz fibration 
of  a K3 surface to $S^{2}$ by gluing two copies of the section 
$s_{2,3,7} : D^{2} \to X_{2,3,7}$ on their boundaries. 

Now we see that from each copy of $X_{2,3,7}$, we have 
$\mathcal{T}(2,3,7)\oplus \langle [T^{2}] \rangle$, while in a K3 surface, 
two $[T^{2}]$'s coincide, but instead we have a new 
2-cycle $S$ which does not intersect two copies of  $\mathcal{T}(2,3,7)$ 
but does the regular fiber $[T^{2}]$ once,  
and thus $[T^{2}]$ and $S$ form an intersection form $H$. 
Hence the intersection form of the second integral homology 
of a K3 surface is isomorphic to $2\mathcal{T}(2,3,7)\oplus H$.  
It is also easy to check the following, 
see \S \ref{Subsec:Intersection}. 
\begin{assertion}  $1)$  $\mathcal{T}(2,3,7)=E_{10}$ is isomorphic to 
$E_{8}\oplus H$. (See \ref{Subsec:Intersection}). 
\\ 
$2)$  $E_{k}$ is unimodular if and only if $k=8$ or $k=10$.   
\\
$3)$  $\mathcal{T}(p,q,r)$ for $1/p+1/q+1/r \leqq 1$ is unimodular 
if and only if $(p,q,r)=(2,3,7)$.  
\end{assertion}
\begin{proposition}  
For the other extended strange duality pair than $T_{2,3,7}$ - $T_{2,3,7}$, 
there exists a section to the Lefschetz fibration 
constructed above, but it always intersects with some cycles 
of $\mathcal{T}(p,q,r)$ and $\mathcal{T}(p',q',r')$.   
\end{proposition}
\begin{proof}  
Otherwise, we have a section $S$ which does not intersect with 
any of the cycles of $\mathcal{T}(p,q,r)$ and $\mathcal{T}(p',q',r')$.  
Then it is easily seen form the Meyer-Vietoris argument 
that the lattice of a K3 surface is isomorphic to 
 $\mathcal{T}(p,q,r)\oplus\mathcal{T}(p',q',r')\oplus H$
where $H$ is the intersection form of $\langle T^{2}, S\rangle$. 
Then  $\mathcal{T}(p,q,r)$ and $\mathcal{T}(p',q',r')$ 
are unimodular.  This is a contradiction.   
\end{proof}
Only in the case of $T_{2,3,7}$ - $T_{2,3,7}$,  
we have such a section and 
recover the K3 lattice $2E_{8}\oplus 3H$, while two $H$'s are from 
 $\mathcal{T}(2,3,7)$ and the third $H$ is for the section $S$ and the fiber.  
Among six generators of three $H$'s, only $S$ seems to be able to be 
represented by an embedded two-sphere. 
\end{example}

\subsection{Inose fibration}\label{subsec:Inose}

There exists a loop $L$ on the base space of 
a generic elliptic fibration $\Phi$ of a K3 surface 
such that $\Phi^{-1}(L)$ is a Sol-manifold. 
Indeed, we have shown that for any 
extended strange duality pair of cusp singularities 
there exists a disk $D\subset \CP^1$ such that 
$\Phi^{-1}(D)$ and its exterior are 
diffeomorphic to their Milnor fibers. 
Thus we may take $L=\partial D$. 
However, we have yet to obtain specific examples of such loops on individual elliptic K3 surfaces.  
Further, there is no such loop in general 
for a non-generic elliptic K3 surface. 
\par
For example, consider the product $C_1\times C_2(\subset \CP^2\times \CP^2)$ of elliptic curves 
\[
C_j: {y_j}^2={x_j}^3+a_j{x_j}+b_j
\quad (4{a_j}^3+27{b_j}^2\neq0,\quad j=1,2)
\]
having the involution $\iota:((x_1,y_1),(x_2,y_2))\mapsto ((x_1,-y_1),(x_2,-y_2))$ with 
16 fixed points $(N_{1,k}, N_{2,l})$ ($k,l \in\{1,2,3,\infty\}$). 
Here $\{N_{j,1}, N_{j,2}, N_{j,3}\}=C_j\cap \{y_j=0\}$ presents 
the three distinct solutions of the cubic equation ${x_j}^3+a_j{x_j}+b_j=0$, 
and $N_{j,\infty}\in C_j$ is the point at infinity ($j=1,2$). 
Blowing up the corresponding 16 nodes of the quotient $(C_1\times C_2)/\iota$, 
we obtain a K3 surface $\textrm{Km}(C_1\times C_2)$ which is called 
the Kummer surface of the abelian surface $C_1\times C_2$ (with respect to $\iota$). 
Then the projection of $\textrm{Km}(C_1\times C_2)$ to the $x_j$-axis is 
an elliptic fibration with four singular fibers of type $\textrm{I}_0^*$ in Kodaira's classification. 
Since the monodromy of the type $\textrm{I}_0^*$ singular fiber is 
$\displaystyle \begin{pmatrix} -1&0 \\ 0&-1 \end{pmatrix}$, we see that 
there is no loop with hyperbolic monodromy. 
On the same K3 surface, Inose \cite{Inose} found that 
\[
\Phi_{\textrm{Inose}}=\frac{y_2}{y_1}\left(=\frac{-y_2}{-y_1}\right): \textrm{Km}(C_1\times C_2)\to \CP^1
\] 
defines another elliptic fibration, which we call the Inose fibration. 
Hereafter we fix the parameters 
as $a_1=-3$, $b_1=0$, $a_2=3$, $b_2=2$. Let $\Sigma_P$ denote 
the fiber $\Phi_{\textrm{Inose}}^{-1}(P)$ at each point $P\in \CP^1$. Then we see that  
\[
\Sigma_P=
\left\{ ~ \frac{{x_2}^3+3{x_2}+2}{{x_1}^3-3{x_1}}=P^2,\quad 
\frac{y_2}{y_1}=P ~ \right\}\subset \textrm{Km}(C_1\times C_2)
\] 
is regular unless $P=0$, $P=\infty$ or 
$P$ is one of the 8 roots $(-4)^{1/8}$. In the case where $P=0$, 
according to the three interpretations $\displaystyle P=\frac{*}{\infty}$, 
$\displaystyle \frac{0}{*}$, 
$\displaystyle \frac{0}{\infty}$ of $P=0$, we have $1+3+3$ irreducible components of 
the singular fiber $\Sigma_0$, namely, the line $\{N_{1,\infty}\}\times \CP^1$, 
the three lines $\CP^1\times \{N_{2,1},N_{2,2},N_{2,3}\}$, and 
the three blown-up lines which are originally the cross points $\{N_{2,1},N_{2,2},N_{2,3}\}$ of the preceding lines in $(C_1\times C_2)/\iota$. 
This implies that $\Sigma_0$ is of type $\mathrm{IV}^*$ in Kodaira's classification. 
Similarly, $\Sigma_\infty$ is of type $\mathrm{IV}^*$. 
It is easy to see that the other singular fibers are of type $\mathrm{I}_1$, i.e., of Lefschetz-type. 
Looking at the monodromy around these singular fibers in detail, 
we can obtain, for example, the following result. 
The detail of its proof and further investigations will be discussed elsewhere. 
\begin{proposition} 
There exists a star-convex domain $(D,0)\subset \C=\CP^1\setminus\{\infty\}$ such that $\Phi_{\textrm{Inose}}^{-1}(\partial D)$ 
is diffeomorphic to the link of each of the following cusp singularities; 
two self dual ones $T_{2,3,7}$, $T_{2,5,5}$; and both in the dual pair $\{T_{2,4,5}, T_{2,3,8}\}$.  
\end{proposition}
\begin{proof} 
Suppose that the boundary of a star-convex domain $(D,0)\subset \C$ is 
a smooth curve that avoids the 8 roots $(-4)^{1/8}$. 
Note that each quadrant of $\C=\R^2$ has two of the 8 roots. 
Let $c_l$ denote the number of the roots in the $l$-th quadrant 
which belongs to $D$ ($l=1,2,3,4$). 
We will show the following implications. 
\begin{enumerate}
\item $(c_1,c_2,c_3,c_4)=(0,0,2,2)\Rightarrow \Phi_{\textrm{Inose}}^{-1}(\partial D)$ is diffeomorphic 
to $\partial X_{2,3,7}$. 
\item $(c_1,c_2,c_3,c_4)=(0,2,0,2)\Rightarrow \Phi_{\textrm{Inose}}^{-1}(\partial D)$ is diffeomorphic 
to $\partial X_{2,5,5}$.
\item $(c_1,c_2,c_3,c_4)=(0,1,0,2)\Rightarrow \Phi_{\textrm{Inose}}^{-1}(\partial D)$ is diffeomorphic 
to $\partial X_{2,4,5}$.
\item $(c_1,c_2,c_3,c_4)=(0,2,1,2)\Rightarrow \Phi_{\textrm{Inose}}^{-1}(\partial D)$ is diffeomorphic 
to $\partial X_{2,3,8}$.
\end{enumerate}
To show them, we 
move the point $P$ along a loop $L$ on $\C$, and 
watch the movement of the six solutions of 
$P^2({x_1}^3-3{x_1})=2\pm2\sqrt{-1}$ which are 
the critical values of the natural projection of $\Sigma_P$ to the $x_1$-axis. 
We fix the base point of $L$ as a small positive real number $\e$. 
Then we may regard approximately the initial position of the six points 
as $((2\pm 2\sqrt{-1})/(\e^2))^{1/3}$. 
Among them, the three points $((2+ 2\sqrt{-1})/(\e^2))^{1/3}$ form a 
counterclockwise triangle 
whose edges are the images of essential simple closed curves 
$\alpha$, $\beta$, $\gamma\subset \Sigma_\e$ in this order, 
where the image of $\alpha$ is the edge starting from the first quadrant. 
Each pair of $\alpha$, $\beta$, $\gamma$ meets at a single point, 
and thus represents a generator of $H_1(T^2)$. 
The other three points $((2-2\sqrt{-1})/(\e^2))^{1/3}$ provide 
three curves $\alpha'$, $\beta'$, $\gamma'\subset \Sigma_\e$
which are respectively disjoint from 
(i.e., parallel to) $\alpha$, $\beta$, $\gamma$. 
We take $L$ as an economic counterclockwise loop around 
a first quadrant element of $(-4)^{1/8}$, where an economic loop is 
going straight there; turning small; 
and coming straight back. 
Then the monodromy is (isotopic to) 
the right-handed Dehn-twist $\tau_\alpha$ along $\alpha$. 
Next we take $L$ so that it starts with $k\pi/2$ rotation with radius $\e$; goes around a $(k+1)$-th quadrant element of $(-4)^{1/8}$ economically; 
and ends with $-k\pi/2$ rotation with radius $\e$ ($k=1,2,3$). 
Then the monodromy is $\tau_\gamma$ ($k=1$), 
$\tau_\beta$ ($k=2$), or $\tau_\alpha$ ($k=3$). 
If we take $L$ as the $2\pi$ rotation with radius $\e$, 
the six critical points rotate by $-2\pi/3$ around $0$, 
and therefore the monodromy can be written as $(\tau_a\tau_\beta)^4$. 
Now we see that the monodromy along $\partial D$ is 
the composition 
$(\tau_\alpha\tau_\beta)^4
\tau_\alpha^{c_4}\tau_\beta^{c_3}
\tau_\gamma^{c_2}\tau_\alpha^{c_1}$. 
We have the above implications by straightforward 
calculation, which we omit here.
\end{proof}
We should notice that the topology of the elliptic fibration near the type $\textrm{IV}^*$ singularity 
is well-known. Particularly, from the result of Naruki \cite{Naruki}, we can see that 
the fibration near the singular fiber can be deformed into one isomorphic to 
our Lagrangian fibration in \S 2 if we put $(p,q,r)=(2,3,3)$ formally. 
However, we have not yet explored the topology of 
the pieces of the above decompositions of Inose fibration,  
whether they provide the Milnor fibers or not. 

\section{\large Foliated Lefschetz fibration} 
\label{Section;Foliated LF}
\subsection{Lawson type foliations}

In the early 1950s, Reeb constructed the first example of a codimension-one foliation on the $3$-sphere, now known as the {\em Reeb foliation}. The construction was carried out by gluing together two copies of {\em Reeb components} (foliated solid tori) along their boundaries according to the genus-one Heegaard splitting of $S^3$. Since then, the construction of explicit examples of codimension-one foliations on odd-dimensional spheres has become a central problem in foliation theory. 
About twenty years later, Lawson \cite{Law} constructed the first example of codimension-one foliation on the $5$-sphere using the Milnor fibration associated with the $\tilde{E}_6$ singularity. The key point of his construction was that the foliation on the Reeb component can be pulled back to a tubular neighborhood of the link of the $\tilde{E}_6$ singularity via a submersion, since this link fibers over $S^1$. 
Since simple elliptic and cusp singularities also have the same feature (see Theorem~\ref{chara}), one can obtain a codimension-one foliation on $S^5$ using any of these singularities instead of the $\tilde{E}_6$ singularity. 
We call these the {\em Lawson type foliations} on $S^5$. 
Based on Lawson's method, Tamura \cite{Tamura} succeeded in giving an example on every odd-dimensional sphere by constructing some explicit open book decomposition of the sphere.  
Soon after that, Thurston \cite{Thurston} proved a general existence theorem for codimension-one foliations (Thurston's $h$-principle), and as a result, the construction of explicit codimension-one foliations lost its original motivation. 
However, in recent years, as indicated by the works of the third and fourth authors \cite{Mori, Mi1, Mori2, Mi2}, the significance of the Reeb foliation and the Lawson type foliation has been revisited from the viewpoint of geometric structures such as contact and symplectic structures.


In this section, we show that a Lawson type foliation admits 
a {\em foliated Lefschetz fibration} over the Reeb foliation on $S^{3}$, namely, 
there exists a leafwise Lefschetz fibration between these foliations that is transverse to the Reeb foliation. 
Consequently, all Lawson type foliations can be regarded as the pullbacks of the Reeb foliation. 
In addition, we obtain an alternative proof of the third author's result 
that a Lawson type foliation admits a leafwise symplectic structure. 

\subsubsection{3-dimensional Reeb foliation and Reeb component}
First let us recall the Reeb foliation on $S^3$. 
We regard $S^3$ as the unit hypersphere in $\C ^2$, namely, 
$$S^{3}=\{(w_{1}, w_{2}) \in \C^{2}; |w_{1}|^{2}+|w_{2}|^{2}=1\}.$$
Then we split it into two solid tori $R_{1}=\{|w_{1}|^{2}\geqq 1/2\}$ and $R_{2}=\{|w_{2}|^{2}\geqq 1/2\}$. 
For each $j$ ($j=1,2$), the diffeomorphism between $R_j$ and $D^2\times S^1$ is given as follows:
$$\psi _1\colon R_1\to D^2\times S^1; (w_1, w_2)\mapsto (w_2/\sqrt{2}, \arg w_1), $$
$$\psi _2\colon R_2\to D^2\times S^1; (w_1, w_2)\mapsto (w_1/\sqrt{2}, \arg w_2).$$
Now we explain that the solid torus $D^2\times S^1$ admits a codimension-one foliation whose boundary torus is the only compact leaf. 
Consider the foliation on $D^2\times \R$ such that the boundary $S^1\times \R$ is one leaf 
and any other leaf is given as the graph of the function $$\Phi_c(u, v)=\frac{u^2+v^2}{1-u^2-v^2}+c,$$ 
where $(u,v)\in \mathrm{Int} {D^2}$ and $c$ is a real constant. 
Then, this foliation is translation invariant, and in particular, invariant under 
the action of $\Z$ given by $$n\cdot (u, v, t)=(u, v, t+n) \; (n\in \Z). $$
Hence, it induces a codimension-one foliation $\mathcal{F}$ on $D^2\times S^1=(D^2\times \R)/\Z$. 
The solid torus $D^2\times S^1$ equipped with the foliation $\mathcal {F}$ is called the {\em Reeb component}. 
Pulling this foliation back to each solid torus $R_j$ via the diffeomorphism $\psi _j$ ($j=1, 2$), and then gluing them together along their boundary tori, we obtain a codimension-one foliation $\mathcal{F}_R$ on $S^3$ whose only compact leaf is the torus 
$$\partial R_1=\partial R_2=\{|w_1|=|w_2|=1/\sqrt{2}\}. $$
This foliation is called the {\em Reeb foliation}.

\begin{figure}[h]
\centering
\includegraphics[width=70mm]{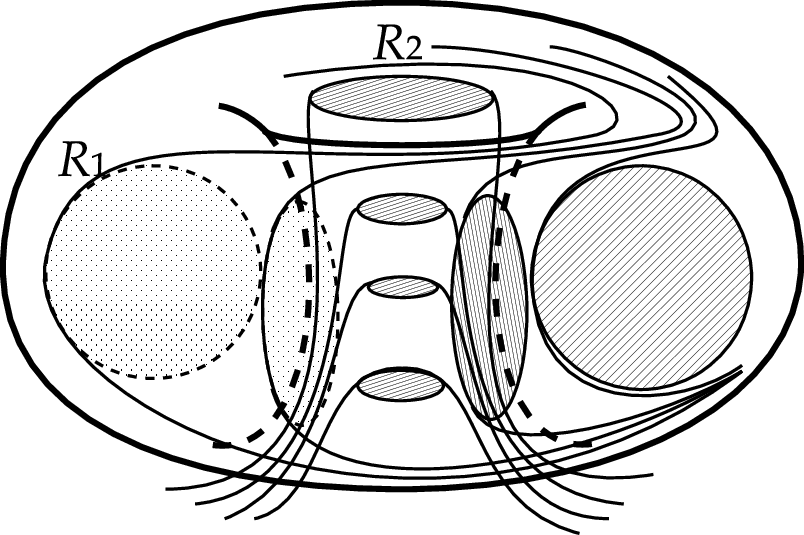}
\caption{The Reeb foliation}
\label{reeb}
\end{figure}

\subsubsection{Lawson type foliations}
 
We review the Lawson type foliations associated with $T_{p,q,r}$-singularities with $1/p+1/q+1/r\leqq 1$.  
As seen in \S~\ref{sing}, the link $L$ of such a singularity is a $T^{2}$-bundle over the circle $S^{1}$. 
We denote the $T^2$-bundle by $\pi \colon L\to S^1$. 
Let $f$ be the defining polynomial of the $T_{p,q,r}$-singularity as in \S~\ref{Milnor fiber}. 
Since $0$ is a regular value of the function $f$ restricted to $S^5$, 
$f^{-1}(D^2_{\rho})\cap S^5$ is a product-type tubular neighborhood of $L$ in $S^5$ for a sufficiently small positive number $\rho$, where $D^2_{\rho}=\{ w\in \C; |w|\leq \rho\}$. 
Thus there exists a diffeomorphism $$\phi \colon f^{-1}(D^2_{\rho})\cap S^5\to L\times D^2_{\rho}$$ 
such that the composition $\mathrm{pr}_2\circ \phi $ coincides with $f$, where $\mathrm{pr}_2$ denotes the projection to the second factor. 
Then we set $U_{\delta }=\phi ^{-1}(L\times D^2_{\delta }) (=f^{-1}(D^2_{\delta })\cap S^5)$ for any positive number $\delta $ with $\delta \leq \rho $. 

Milnor's fibration theorem (Theorem~\ref{Milnor}) says that the function $\arg f$ yields an open book decomposition of $S^5$ whose binding coincides with $L$. Each page $(\arg f)^{-1}(\theta )$ ($\theta \in S^1$) transversely intersects with $\partial U_{\delta }$ for any $\delta $ with $0<\delta \leq \rho $. Hence, it gives a codimension-one foliation $\mathcal{F}''_{p,q,r}$ on $S^5\setminus U_{\rho}$ that is transverse to the boundary $\partial U_{\rho}$. 
Now we consider the $1$-dimensional foliation $\mathcal{L}$ on the annulus $$D^2_{\rho }\setminus D^2_{\rho/2}=\{w\in \C; \rho/2<|w|\leq \rho \}$$ that is radial near $\partial D^2_{\rho }$ and wraps around $\partial D^2_{\rho/2}$. 
Pulling it back by $f$, we obtain a codimension-one foliation $f^{\ast }\mathcal{L}$ on $U_{\rho }\setminus U_{\rho/2}$ that smoothly matches with $\mathcal{F}''_{p,q,r}$. Then these two foliations form a codimension-one foliation $\mathcal{F}'_{p,q,r}$ on $S^5\setminus U_{\rho/2}$ whose leaves wrap around the boundary $\partial U_{\rho/2}$.

Finally, we pull back the Reeb component by the $T^2$-bundle map 
$$ (\pi \times \mathrm{id})\circ \phi  \colon U_{\rho/2}\to L\times D^2_{\rho/2}\to S^1\times D^2_{\rho/2},$$ 
where $S^1\times D^2_{\rho/2}$ is identified with $D^2\times S^1$ via the natural diffeomorphism. 
Then, $\mathcal{F}'_{p,q,r}$ and $((\pi \times \mathrm{id})\circ \phi )^{\ast}(\mathcal{F})$ smoothly matches along $\partial U_{\rho/2}$ to form a codimension-one foliation $\mathcal{F}_{p,q,r}$ on $S^5$, whose only compact leaf is $\partial U_{\rho/2}\cong L\times S^1$. This foliation is called the {\em Lawson type foliation} associated with a $T_{p,q,r}$-singularity.\\

\vspace{1pt}

\begin{remark}~\label{turbulize}
Let $M$ be a manifold with boundary. 
Any codimension-one foliation on $M$ whose leaves are all transverse to the boundary $\partial M$ 
can be modified only on a collar neighborhood of $\partial M$ to be a foliation wrapping around $\partial M$. 
Such an operation is called a {\em turbulization}. 
The Lawson type foliation $\mathcal{F}_{p,q,r}$ is obtained from the Milnor fibration associated with a $T_{p,q,r}$-singularity 
by pulling back the Reeb component to a tubular neighborhood of the link $L$ and turbulizing all the Milnor fibers around it. 
In fact, the Reeb component can be also regarded as the resultant foliation of a turbulization of the trivial codimension-one foliation on $D^2\times S^1$. From this point of view, the Reeb foliation is the resultant foliation of applying Lawson's method to the trivial open book decomposition on $S^3$ given by the function $\arg w_2$. 
\end{remark}

\subsection{Foliated Lefschetz fibrations}\label{Subsec:FolLF}

Now we introduce the following notion, which was suggested by Presas. 

\begin{definition}[Presas \cite{Presas}] 
\label{Def:FolLF}
Let $X$ and $Y$ be a smooth 5-manifold and 3-manifold, 
$\F_{X}$ and $\F_{Y}$ be smooth foliations of codimension one 
on  $X$ and $Y$ respectively.  
A smooth foliation preserving map 
$\Psi : X \to Y$ is said to be a
{\em foliated Lefschetz fibration} 
if it satisfies the following conditions. 
\begin{itemize}
\item[(i)]
$\Psi$ is transverse to  $\F_{Y}$, 
\IE the composition of $D \Psi : TX \to TY$ 
and $TY \twoheadrightarrow \nu\F_{Y}$ is everywhere surjective, 
where $\nu \cdot$ denotes the normal bundle. 
\item[(ii)] 
On each leaf $L$ of $\F_{X}$, $\Psi\vert_{L}:L \to \Psi(L)$ is 
a Lefschetz fibration between the leaves. 
\item[(iii)]
The set of critical points of Lefschetz fibrations between leaves 
forms a complete 1-dimensional smooth submanifold which is 
transverse to $\F_{X}$.   
\end{itemize}
\end{definition}
\begin{remark} 
By taking a small perturbation if necessary,  
we may assume that any singular fiber of a foliated Lefschetz fibration contains only one critical point. 
Under this assumption, the critical locus embeds into $Y$ as a knot or link transverse to $\F_{Y}$.  
\end{remark}

A nontrivial example of a foliated Lefschetz fibration is given by the following result, which follows from our Main Theorem. 

\begin{theorem}[Foliated Lefschetz Fibration]
\label{Thm:FolLF}
For a $T_{p,q,r}$ singularity with $1/p+1/q+1/r\leqq 1$, 
the associated Lawson type foliation  $\F_{p,q,r}$ on $S^{5}$ 
admits a foliated Lefschetz fibration over the 
Reeb foliaiton $\F_{R}$ on $S^{3}$,   
whose critical locus consists of three components which 
are embedded into $S^{3}$ as torus knots of type 
{\color{black} $(p,p-1)$, $(q,q-1)$,  and $(r,r-1)$}.
\end{theorem}
\begin{proof}
Let $K$ be the trivial knot in $S^3=\{(w_{1}, w_{2}) \in \C^{2}; |w_{1}|^{2}+|w_{2}|^{2}=1\}$ defined by $w_2=0$. 
We recall that $\arg f \colon S^5\setminus L\to S^1$ and $\arg w_2\colon S^3\setminus K\to S^1$ 
define open book decompositions, which are described by $(S^5, \arg f)$ and $(S^3, \arg w_2)$, respectively. 
As is mentioned in Remark~\ref{turbulize}, 
the Lawson-type foliation $\F_{p,q,r}$ and the Reeb foliation $\F_{R}$ 
are both the resultant foliations of the following procedure to 
the open book decompositions $(S^5, \arg f)$ and $(S^3, \arg w_2)$, respectively. 
\begin{enumerate}
\item
Pull back the Reeb component to a tubular neighborhood of the binding by a submersion. 
\item
Turbulize all the pages around this tubular neighborhood. 
\end{enumerate}
Therefore, it is enough to construct a map $\Psi \colon S^5\to S^3$ satisfying the following conditions:
\begin{enumerate}
\item[(a)]
$\Psi (S^5\setminus L)=S^3\setminus K$ and $\Psi (L)=K$.
\item[(b)]
$\Psi |_{S^5\setminus L}\colon S^5\setminus L\to S^3\setminus K$ is a foliated Lefsctez fibration with respect to the foliations defined as the level sets of $\arg f$ and $\arg w_2$. 
\item[(c)]
$\Psi |_{L}\colon L\to K$ coincides with the $T^2$-bundle map $\pi \colon L\to S^1$, 
where $K$ is idetified with $S^1$ by the diffeomorphism $K\ni (w_1, 0)\mapsto w_1\in S^1$. 
\end{enumerate} 
In \S~\ref{construction}, we have already obtained an $S^1$-parametric Lefschetz fibration 
$$(g, h)\colon \bigcup_{\theta \in S^1}Y_{\theta }\to D^2_{\frac{1}{3}}\times S^1_{\frac{1}{a}},$$
where the map $g\colon Y_{\theta }\to D^2_{\frac{1}{3}}$ is the Lefschetz fibtarion constructed in Theorem~\ref{precise statement} 
and the fiber bundle $h\colon \bigcup_{\theta \in S^1}\mathrm{Int} Y_{\theta }\to S^1_{\frac{1}{a}}$ is isomorphic to the Milnor fibration $\arg f \colon S^5\setminus L\to S^1$. 
Hence, there exists a smooth map $\tilde{g}\colon S^5\to D^2$ such that 
$$(\tilde{g}, \arg f)\colon S^5\setminus L\to \mathrm{Int} D^2\times S^1$$
is an $S^1$-parametric Lefschetz fibration and $\tilde{g}|_{L}\colon L\to \partial D^2=S^1$ coincides with the $T^2$-bundle $\pi $. 
Notice that the open solid torus $\mathrm{Int} D^2\times S^1$ with trivial codimension-one foliation 
can be seen as $S^3\setminus K$ foliated by the level sets of the function $\arg w_2$. 
Let us denote the diffeomorphism by $\Phi \colon \mathrm{Int} D^2\times S^1\to S^3\setminus K$. 
Now we define the map $\Psi \colon S^5\to S^3$ by 
\begin{equation*}
\Psi (p)=
\begin{cases}
\Phi (\tilde{g}(p), \arg f(p))\in S^3\setminus K \;\; \text{if} \; p\in S^5\setminus L, \\
(\tilde{g}(p), 0)\in K\subset \C^2 \;\; \text{if} \; p\in L. 
\end{cases}
\end{equation*}
Then it satisfies (a), (b) and (c) by construction, and thus, 
we obtain a foliated Lefschetz fibration $\Psi \colon (S^5, \mathcal{F}_{p,q,r})\to (S^3, \mathcal{F}_R)$. 

A perturbation of $g$ like in \S 3 naturally implies that 
the image of the critical locus by $\Psi$ 
consists of three torus knots of type $(p,p-1)$, $(q,q-1)$,  and $(r,r-1)$.
\end{proof}

Notice that $\Psi \colon S^5\to S^3$ is a $T^2$-fibration, 
by which the Lawson type foliation $\F_{p,q,r}$ can be regarded as the pullback of the Reeb foliation $\F_{R}$. 

In fact, the notion of foliated Lefschetz fibration was originally motivated by an attempt 
to construct leafwise symplectic foliations based on Gompf's construction of symplectic structures from Lefschetz fibrations. 
Now we see that this attempt has been successfully completed; indeed,   
the following result of the third author can be reproven as a corollary to Theorem~\ref{Thm:FolLF}. 

\begin{corollary}[\cite{Mi1, Mi2}]\label{Cor:Poisson}
The Lawson-type foliation $\F_{p,q,r}$ on $S^{5}$ 
for  $1/p+1/q+1/r\leqq 1$ 
admits a leafwise symplectic structure. 
In other words, it is the four dimensional symplectic foliation 
of a regular Poisson structure on the 5-sphere. 
\end{corollary}

\begin{proof}
As the Reeb foliation $\F_{R}$ admits a leafwise 
symplectic structure, it suffices to apply Gompf's theorem  \cite{Gompf}
in the foliated (parametric) situation. 
\end{proof}

\appendix
\section{Concise proof for the case where $p, q, r\geq 3$}
In this appendix, we provide a concise proof of Theorem~\ref{thm: main 2} under the conditions  $p, q, r \geq  3$ and $t=0$. 
We suppose that $f$, $a$ and $V_a(\e , t)$ are as in \S~\ref{Milnor fiber}, and $g$ as defined at the beginning of \S~\ref{CLF}. \\

\begin{theorem}~\label{original}
Suppose $3\leq p\leq q\leq r$ and $a>3r$. 
Then there exists a positive number $\delta $ such that 
if $|w|<\delta $, then the map $g|_{V_a(1,w)}\colon V_a(1,w)\to \C$
has exactly $(p+q+r)$ critical points 
$$\big(w^{\frac{1}{p}}{(u_p)}^j, 0, 0\big), \; 
\big(0, w^{\frac{1}{q}}{(u_q)}^k, 0\big), \; \big(0,0, w^{\frac{1}{r}}{(u_r)}^l\big),$$
where $u_n=\exp {(\frac{2\pi i}{n})}$, 
$0\leq j \leq p-1$, $0\leq k \leq q-1$, and $0\leq l \leq r-1$. 
Moreover, the $2$-jet of each of these critical points coincides with that of a Lefschetz singularity. 
\end{theorem}

The statement of this theorem is weaker than that of Theorem~\ref{thm: main 2}. 
However, it allows for a simpler proof, 
which may help the reader better understand the original proof of Theorem~\ref{thm: main 2}.  
Here we note that in the case $(p,q,r)=(2,3,r)$, the statement corresponding to Theorem~\ref{original} does not hold. 
In order to support all the simple elliptic and cusp singularities including the case where $p=2$, 
we need to take a larger $a$ and choose $w$ so that $|w|$ is not too small relative to $a$ 
(see Theorem~\ref{p=2} and Theorem~\ref{thm: main 2}). 
This is why the evaluation arguments become rather complicated in the proof of Theorem~\ref{thm: main 2}. 
We will explain it in detail later in \S~\ref{case p=2}. \\

To prove Theorem~\ref{original}, we first recall some notation in \S~\ref{CLF}. We set 
\[x=x_1+ix_2, \; y=y_1+iy_2, \; z=z_1+iz_2 \;\; (x_1,x_2,y_1,y_2, z_1, z_2\in \R). \]
Using holomorphic and anti-holomorphic vectors 
\[ 
\frac{\partial }{\partial x}=\frac{1}{2}\left(\frac{\partial }{\partial x_1}-i\frac{\partial }{\partial x_2}\right), \; 
\frac{\partial }{\partial \bar{x}}=\frac{1}{2}\left(\frac{\partial }{\partial x_1}+i\frac{\partial }{\partial x_2}\right), 
\]
we define two smooth vector fields $e_x$ and $E_x$ on $\C$ by 
\begin{eqnarray*}
e_x=i\left(x\frac{\partial }{\partial x}-\bar{x}\frac{\partial }{\partial \bar{x}}\right), \;
E_x=x\frac{\partial }{\partial x}+\bar{x}\frac{\partial }{\partial \bar{x}}. 
\end{eqnarray*}
Using the real coordinates $(x_1, x_2)$, they are describe as 
\begin{eqnarray*}
e_x=x_1\frac{\partial }{\partial x_2}-x_2\frac{\partial }{\partial x_1}, \; 
E_x=x_1\frac{\partial }{\partial x_1}+x_2\frac{\partial }{\partial x_2}.  
\end{eqnarray*}
Hence, these are the rotational vector field and the Euler vector field on $\C$, respectively. 
They canonically extend over $\C^3$ as smooth vector fields, 
which we denote by the same symbols $e_x$ and $E_x$. 
The vector fields $e_y$, $e_z$, $E_y$ and $E_z$ are similarly defined. 
Moreover, we define the vector field $e_0$ on $(\C^{\ast })^3$ by 
\[ e_0=\frac{1}{|x|^2}E_x+\frac{1}{|y|^2}E_y+\frac{1}{|z|^2}E_z. \]

Now we study the structure of level sets of the function 
$$g(x,y,z)=|x|^2+e^{\frac{2\pi i}{3}}|y|^2+e^{\frac{4\pi i}{3}}|z|^2,$$
which is necessary for detecting all the critical points of the map $g|_{V_a(1, w)}$. 
By an explicit computation, we can easily check that the singular set $\Sigma (g)$ of the function $g$ 
is the union of the $x$-axis, the $y$-axis and the $z$-axis. 
Hence, the intersection of any level set of $g$ and $(\C^{\ast })^3$ is a smooth real $4$-dimensional manifold. 
We put $W_c=g^{-1}(c)\cap (\C^{\ast })^3$ for any $c\in \R$. \\

\begin{proposition}~\label{tangent}
For any point $b$ in $W_c$, $\{e_x, e_y, e_z, e_0\}$ is a real basis of the tangent space $T_bW_c$. 
\end{proposition}
\begin{proof}
Since $e_x$ is a rotational vector field, the function $|x|^2$ is constant along its flow line, and thus, $e_x(|x|^2)=0$. 
Moreover, it is clear that $e_x(|y|^2)=e_x(|z|^2)=0.$ Hence we have $e_x(g)=0$. 
By the same argument, it follows that $e_y(g)=e_z(g)=0$. 

By the equalities $E_x(|x|^2)=2|x|^2$, $E_y(|y|^2)=2|y|^2$, $E_z(|z|^2)=2|z|^2$, we obain  
\[ 
e_0(g)=\frac{1}{|x|^2}E_x(|x|^2)+\frac{e^{\frac{2\pi i}{3}}}{|y|^2}E_y(|y|^2)+\frac{e^{\frac{4\pi i}{3}}}{|z|^2}E_z(|z|^2)
=2(1+e^{\frac{2\pi i}{3}}+e^{\frac{4\pi i}{3}})=0. \]
Thus we have $e_x(g)=e_y(g)=e_z(g)=e_0(g)=0$, namely, $e_x, e_y, e_z, e_0\in T_pW_c$. 
Since these $4$ vectors are linearly independent over $\R$, 
they form a basis of the real $4$-dimensional vector space $T_pW_c$. 
\end{proof}

Based on this proposition, we give a necessary and sufficient condition 
for a point on the Milnor fiber $V_a(1, w)$ being a regular point of the restricted map $g|_{V_a(1, w)}$. \\

\begin{proposition}~\label{regular}
A point $b=(x,y,z)\in V_a(1, w)$ with $xyz\ne 0$ is a regular point of the complex-valued function $g|_{V_a(1, w)}$ 
if and only if the equality \[\dim _{\R}\big<e_x(f), e_y(f), e_z(f), e_0(f)\big>_{\R}=2\] holds at the point. 
\end{proposition}
\begin{proof}
We put $c=g(b)$. 
The point $b$ is a regular point of $g|_{V_a(1, w)}$ if and only if 
the Milnor fiber $V_a(1, w)$ and the level set $W_c=g^{-1}(c)$ of $g$ transversely intersect at $b$. 
This condition is also equivalent to 
the one that $b$ is a regular point of $f|_{W_c}$, 
since $V_a(1,w)$ is a level set of the polynomial function $f$. 
Since we have $$T_pW_c=\langle e_x, e_y, e_z, e_0\rangle _{\R}$$ by Proposition~\ref{tangent}, 
the condition is paraphrased as 
\[ \big<e_x(f), e_y(f), e_z(f), e_0(f)\big>_{\R}=\C. \]
This completes the proof. 
\end{proof}

By explicit computations, we easily see that 
\begin{eqnarray*}
e_x(f)=i(px^p+axyz), \; e_y(f)=i(qy^q+axyz), \; e_z(f)=i(rz^r+axyz), \\
e_0(f)=\frac{1}{|x|^2}(px^p+axyz)+\frac{1}{|y|^2}(qy^q+axyz)+\frac{1}{|z|^2}(rz^r+axyz). 
\end{eqnarray*}
Noticing $xyz\ne 0$, we see that the equality 
\[\dim _{\R}\big<e_x(f), e_y(f), e_z(f), e_0(f)\big>_{\R}= 0\]
never holds at the point $(x, y, z)$. 
For, if $e_x(f)=e_y(f)=e_z(f)=0$, then it follows that  
\[ \frac{\partial f}{\partial x}(p)=\frac{\partial f}{\partial y}(p)=\frac{\partial f}{\partial z}(p)=0,  \]
which contradicts the assumption that $b$ is a point on $V_a(1,w)$, and hence, a regular point of $f$. 
Therefore, by combining this with Proposition~\ref{regular}, 
we see that $b\in V_a(1,w)$ is a critical point of $g|_{V_a(1,w)}$ if and only if the equality 
\[\dim _{\R}\big<e_x(f), e_y(f), e_z(f), e_0(f)\big>_{\R}= 1\] holds. 
Moreover, it is equivalent to the condition that 
the three complex numbers $e_x(f)$, $e_y(f)$, $e_z(f)$ are real multiples 
of a common non-zero complex number and the equality $e_0(f)=0$ holds.
Indeed, if there exists a non-zero complex number $u$ 
such that $e_x(f)=c_1u$, $e_y(f)=c_2u$, $e_z(f)=c_3u$ ($c_1,c_2, c_3\in \R$) and $e_0(f)\ne 0$ holds, 
then the two complex numbers $w$ and 
\[ e_0(f)=-i\left(\frac{e_x(f)}{|x|^2}+\frac{e_y(f)}{|y|^2}+\frac{e_z(f)}{|z|^2}\right)
=-i\left(\frac{c_1}{|x|^2}+\frac{c_2}{|y|^2}+\frac{c_3}{|z|^2}\right)u \]
are linearly independent over $\R$, which implies that 
\[ \dim _{\R}\big<e_x(f), e_y(f), e_z(f), e_0(f)\big>_{\R}= 2.\] 
Thus we have obtained the following. \\

\begin{proposition}~\label{critical}
A point $b\in V_a(1, w)$ is a critical point of the map $g|_{V_a(1, w)}$ 
if and only if at the point, $e_x(f)$, $e_y(f)$, $e_z(f)$ 
are real multiples of a common non-zero complex number and $e_0(f)=0$ holds. 
\end{proposition}

Now we are ready to prove Theorem~\ref{original}. 

\begin{proof}[{\bf Proof of Theorem~\ref{original}}]
Let $b=(x,y,z)$ be any point in $V_a(1,w)$. 
First we consider the case $xyz=0$. 
In this case, it is easily proven that $b$ is a critical point of $g|_{V_a(1,w)}$ 
only if two of $x$, $y$ and $z$ vanish. 
Therefore, a critical point is on the intersection of 
the Milnor fiber $V_a(1, w)=f^{-1}(w)\cap D^6$ and the union of $x$-axis, $y$-axis and $z$-axis, 
which consists of the $(p+q+r)$ points 
$$\big(w^{\frac{1}{p}}{(u_p)}^j, 0, 0\big), \; \big(0, w^{\frac{1}{q}}{(u_q)}^k, 0\big), \; \big(0,0, w^{\frac{1}{r}}{(u_r)}^l\big),$$
where $u_n=\exp {(\frac{2\pi i}{n})}$, 
$0\leq j \leq p-1$, $0\leq k \leq q-1$, and $0\leq l \leq r-1$. 
It is easily checked that 
these $(p+q+r)$ points are actually critical points and 
their $2$-jets are of Lefschetz type. 

Now what we only have to prove is that 
if $xyz\ne 0$, then $b=(x,y,z)$ is a regular point of $g|_{V_a(1, w)}$. 
In order for that, it suffices to show that $e_0(f)\ne 0$. 
Considering symmetry, we only discuss the case $|x|\leq |y|\leq |z|$. 
Moreover, we have $|x|, |y|, |z|\leq 1$, since $(x,y,z)\in V_a(1, w)\subset D^6$. 
Then the complex number 
\begin{eqnarray*}
e_0(f)&=&\frac{1}{|x|^2}(px^p+axyz)+\frac{1}{|y|^2}(qy^q+axyz)+\frac{1}{|z|^2}(rz^r+axyz)\\
&=&\left(\frac{px^p}{|x|^2}+\frac{qy^q}{|y|^2}+\frac{rz^r}{|z|^2}\right)
+\left(\frac{1}{|x|^2}+\frac{1}{|y|^2}+\frac{1}{|z|^2}\right)axyz
\end{eqnarray*}
never be equal to $0$. 
Indeed, by the conditions $3\leq p\leq q\leq r$, $0<|x|\leq |y|\leq |z|\leq 1$, $a>3r$ and the triangle inequality, 
it follows that 
\begin{eqnarray*}
\left|\frac{px^p}{|x|^2}+\frac{qy^q}{|y|^2}+\frac{rz^r}{|z|^2}\right|
\leq p|x|^{p-2}+q|y|^{q-2}+r|z|^{r-2}\leq r|x|+r|y|+r|z|\leq 3r|z|, \\
\left| \left(\frac{1}{|x|^2}+\frac{1}{|y|^2}+\frac{1}{|z|^2}\right)axyz \right|
>\frac{1}{|x|^2}|axyz|=a\frac{|y|}{|x|}|z|\geq a|z|>3r|z|. 
\end{eqnarray*}
Therefore, a point $b$ with $xyz\ne 0$ is a regular point of $g|_{V_a(1, w)}$. 
\end{proof}
Notice that in the above proof, the condition $3\leq p\leq q\leq r$ is used for 
the estimates $|x|^{p-2}\leq |x|$, $|y|^{q-2}\leq |y|$ and $|z|^{r-2}\leq |z|$.  

\section{The case where $p=2$}~\label{case p=2}
When $p=2$, the proof of Theorem~\ref{original} in the previous section does not work, since $|x|^{p-2}=|x|^0=1$. 
Moreover, in the case where $(p,q,r)=(2,3,r)$, the corresponding claim itself does not hold. \\

\begin{proposition}~\label{curve}
If $(p,q,r)=(2,3,r)$, then there exists a smooth curve $c\colon [0, 1]\to \C^3$ with $c(0)={\bf 0}$  
satisfying the following conditions; 
\begin{enumerate}
\item
$c((0,1])\cap V(0)=\emptyset$, 
\item
$c((0,1])\subset \{xyz\ne 0 \}$, 
\item
$c(s)$ is a critical point of $g|_{V(1, f(c(s)))}$ for each $s\in (0, 1]$. 
\end{enumerate}
\end{proposition}
\begin{proof}
We define a smooth curve $\gamma (s)=\left(x(s), y(s), z(s)\right)$ $(s\geq 0)$ by 
\[ x(s)=-\frac{as}{2}y(s), \;\; y(s)=a^{-2}\left(3-\sqrt{9-a^4s^2+2ra^2s^{r-2}}\right), \;\; z(s)=s.  \]
Then we have $\gamma (0)={\bf 0}$, 
and there exists a small positive number $\e $ such that if $0<s\leq \e$, then $x(s), y(s), z(s)\in \R\setminus \{0\}$. 
Hence, we have $\gamma (s)\in \R^3\setminus \{xyz= 0\}$ and $e_x(f), e_y(f), e_z(f)$ are all purely imaginary. 
Moreover, by $2x(s)+ay(s)z(s)=0$, we obtain that  
\begin{eqnarray*}
e_0(f)&=&\frac{1}{|x|^2}(2x^2+axyz)+\frac{1}{|y|^2}(3y^3+axyz)+\frac{1}{|z|^2}(rz^r+axyz) \\
&=&\frac{1}{y^2}\left(3y^3-\frac{a^2s^2}{2}y^2\right)+\frac{1}{s^2}\left(rs^r-\frac{a^2s^2}{2}y^2\right) \\
&=&-\frac{1}{2}a^2y^2+3y+rs^{r-2}-\frac{1}{2}a^2s^2=0. 
\end{eqnarray*}
Thus $\gamma (s)$ satisfies the condition of Proposition~\ref{critical}. 
Namely, $\gamma(s)$ is a critical point of $g|_{V(1, f(\gamma (s)))}$ with $xyz\ne 0$. 
Therefore, putting $c(s)=\gamma (\e s)$, the curve $c$ satisfies the desired conditions. 
\end{proof}

The curve in Proposition~\ref{curve} intersects with an arbitrarily thin Milnor tube. 
This implies that for any $\delta >0$, there exists a complex number $w$ 
such that $0<|w|\leq \delta $ and $V(1, w)\cap c([0, 1])\ne \emptyset $. 
Therefore, the claim corresponding to Theorem~\ref{original} does not hold when $(p,q,r)=(2,3,r)$. 
However, we can resolve this difficulty by taking $|w|$ not too small with respect to $a$. 
Concretely, it suffices to take $|w|=a^{-1}$. 
Moreover, by retaking $a$ large enough, we obtain the following theorem, 
which corresponds to a restricted version (the case where $t=0$) of Theorem~\ref{thm: main 2}. \\

\begin{theorem}~\label{p=2}
Suppose that $2\leq p\leq q\leq r$, $\frac{1}{p}+\frac{1}{q}+\frac{1}{r}\leq 1$ and $a>4r^2(2r+3)$. 
Then, for any $\theta \in \R$, 
the map $g|_{V_a\left(1,\frac{1}{a}e^{i\theta }\right)}\colon V_a\left(1,\frac{1}{a}e^{i\theta }\right)\to \C$
has exactly $(p+q+r)$ critical points 
$$\left(a^{-\frac{1}{p}}e^{\frac{i\theta }{p}}{(u_p)}^j, 0, 0\right), \; 
\left(0, a^{-\frac{1}{q}}e^{\frac{i\theta }{q}}{(u_q)}^k, 0\right), \; 
\left(0,0, a^{-\frac{1}{r}}e^{\frac{i\theta }{r}}{(u_r)}^l\right),$$
where $u_n=\exp {(\frac{2\pi i}{n})}$, 
$0\leq j \leq p-1$, $0\leq k \leq q-1$, and $0\leq l \leq r-1$. 
Moreover, the $2$-jet of each of these critical points coincides with that of a Lefschetz singularity. 
\end{theorem}
However, as it is, we do not know whether each critical point is really of Lefschetz type, 
and it is not necessarily the case that the boundary of each Milnor fiber is foliated by regular fiber tori. 
Therefore, in order to eliminate these inconveniences, 
we have constructed a smooth deformation $\{X_t\}_{0\leq t\leq 1}$ 
of the Milnor fiber $X_0:=V_a(1,\frac{1}{a}e^{i\theta })$ 
as a convex symplectic submanifold in $\C ^3$ 
so that the restriction of $g$ to $X_1$ becomes a Lagrangian torus fibration over $D^2$, 
and thus obtained Theorem~\ref{thm: main 2}. 
For the actual construction and proof, the reader is referred to \S 2.3. 

%



\end{document}